\documentclass[11pt, a4paper, reqno, francais]{smfart}

\usepackage[margin=3.5cm]{geometry}
\usepackage{amsfonts, amsthm, amsmath, amscd, amssymb}

\numberwithin{equation}{section}

\newcommand{\Q}{\mathbb{Q}}

\newcommand{\R}{\mathbb{R}}

\newcommand{\N}{\mathbb{N}}
\newcommand{\Z}{\mathbb{Z}}

\newtheorem{thm}{Th\'eor\`eme}
\newtheorem{lem}{Lemme}

\newtheorem{Cor}{Corollaire}

\newcommand{\bd}{\mathbf{d}}
\newcommand{\y}{\mathbf{y}}

\renewcommand{\v}{\mathbf{v}}
\renewcommand{\rho}{\varrho}
\renewcommand{\leq}{\leqslant}
\renewcommand{\le}{\leqslant}
\renewcommand{\geq}{\geqslant}
\renewcommand{\ge}{\geqslant}

\renewcommand{\mod}[1]{\hspace{-2.9mm}\pmod{#1}}

\newcommand{\x}{{\bf x}}

\newcommand{\ma}{\mathbf}
\newcommand{\ben}{\begin{enumerate}}
\newcommand{\een}{\end{enumerate}}
\newcommand{\beq}{\begin{equation}}
\newcommand{\eeq}{\end{equation}}

\newcommand{\ve}{\varepsilon}

\newcommand{\mcal}{\mathcal}
\newcommand{\lab}{\label}

\newcommand{\la}{\lambda}
\newcommand{\colt}[2]{\genfrac{}{}{0pt}{1}{#1}{#2}}
\newcommand{\tolt}[3]{\colt{#1}{\colt{#2}{#3}}}

\renewcommand{\d}{\mathrm{d}}
\renewcommand{\leq}{\leqslant}
\renewcommand{\geq}{\geqslant}

\theoremstyle{definition}
\newtheorem*{ack}{Remerciements}

\newcommand{\mal}{\boldsymbol{\lambda}}

\DeclareMathOperator{\vol}{vol}

\DeclareMathOperator{\Res}{Res}

\newcommand{\cU}{\mathcal{U}}
\newcommand{\cV}{\mathcal{V}}
\newcommand{\cA}{\mathcal{A}}

\newcommand{\by}{\mathbf{y}}

\newcommand{\smod}[1]{\;(\text{mod }#1)}

\newcommand{\mka}{\boldsymbol{\kappa}}
\newcommand{\mve}{\boldsymbol{\varepsilon}}
\newcommand{\mnu}{\boldsymbol{\nu}}
\newcommand{\mmu}{\boldsymbol{\mu}}
\newcommand{\msigma}{\boldsymbol{\sigma}}
\newcommand{\mdelta}{\boldsymbol{\delta}}

\newcommand{\mDelta}{\boldsymbol{\Delta}}

\def\dm{{\textstyle{\frac{1}{2}}}}

\begin{document}

\title{Le probl\`eme des diviseurs pour des formes binaires de degr\'e 4}

\author{R. de la Bret\`eche}
\address{
Institut de Math\'ematiques de Jussieu\\
Universit\'e Paris 7 Denis Diderot\\
Case Postale 7012\\
2, Place Jussieu\\
F-75251 Paris cedex 05\\ 
France}
\email{breteche@math.jussieu.fr}

\author{T.D. Browning}
\address{School of Mathematics\\
University of Bristol\\ Bristol\\ BS8 1TW\\
United Kingdom}
\email{t.d.browning@bristol.ac.uk}

\dedicatory{
\begin{center} \`A la m\'emoire affectueuse de George Greaves\end{center}}

\begin{abstract} Nous \'etudions l'ordre moyen du nombre de diviseurs 
des valeurs de certaines formes
binaires quartiques qui ne sont pas irr\'eductibles sur~$\Q$.
\end{abstract}

\begin{altabstract}
We study the average order of the divisor function, as it ranges over 
the values of
binary quartic forms that are reducible over $\Q$.
\end{altabstract}


\date{\today}
\maketitle

\tableofcontents

\section{Introduction} 
Soient $L_1,L_2\in \Z[\x]$ deux formes lin\'eaires non 
proportionnelles  o\`u $\x=(x_1,x_2)$,   $Q\in
\Z[\x]$ une forme quadratique irr\'eductible sur $\Q $ et  $\mcal{R} 
$ une r\'egion
born\'ee convexe de~$\R^2$. Nous nous proposons d'estimer 
asymptotiquement lorsque $X$ tend vers $+\infty$
la somme
$$
   T(X)
=\sum_{\x\in \Z^2\cap
X\mcal{R}}\tau(L_1(\x) L_2(\x) Q(\x)),
$$
o\`u $X\mcal{R}:= \{X\x: \x\in \mcal{R}\}.$
  Ici,
$\tau$ d\'esigne la fonction nombre de diviseurs.

Cette estimation r\'esultera de l'\'etude de la somme
$$S(X):=S(X;L_1,L_2,Q,\mcal{R})=\sum_{\x\in \Z^2\cap
X\mcal{R}}\tau(L_1(\x))\tau(L_2(\x))\tau(Q(\x)).
$$
Ce travail a certains points communs avec \cite{4linear} o\`u nous 
g\'en\'eralisions le
travail de Heath-Brown \cite{h-b03} concernant les sommes
\begin{equation}\label{sumR}
\sum_{\x\in \Z^2\cap X\mcal{R}}r(L_1(\x))r(L_2(\x))r(L_3(\x))
r(L_4(\x)),
\end{equation}
lorsque les $L_i$ sont des formes lin\'eaires non proportionnelles 
deux \`a deux. Ici la quantit\'e $r(n)$ d\'esigne le nombre de repr\'esentations de $n$
comme une somme de deux carr\'es.

\goodbreak

Dans toute la suite, nous ferons sur les formes $L_1$, $L_2$, $Q$ et 
$\mcal{R}$ les hypoth\`eses
minimales suivantes :
\begin{enumerate}
\item[(H1)]
$\mcal{R}$ est un ouvert born\'e, convexe avec une fronti\`ere 
d\'efinie par une fonction
continuement diff\'erentiable par morceaux,
\item[(H2)]
   $L_1$ et $L_2$ ne sont pas proportionnelles et $Q$ est 
irr\'eductible sur $\Q$,  
\item[(H3)]
$L_i(\x)>0$ et $Q(\x)>0$ pour tout  $\x
\in \mcal{R}$,
\end{enumerate} et nous noterons
\begin{equation}\begin{split}
L_1(\x)&:=a_1x_1+b_1x_2,\\
L_2(\x)&:=a_2x_1+b_2x_2,
\\ Q(\x)&:=a_3x_1^2+b_3x_2^2+c_3x_1x_2,
\end{split}\label{defforme}\end{equation}
avec 
\begin{equation}\label{defDelta}
\Delta={\rm disc}( Q)=c_3^2-4a_3b_3.\end{equation}\goodbreak

Nous introduisons les notations, lorsque $\ma{d}=(d_1,d_2,d_3)\in\N^3$,
\begin{equation}\label{eq:farm} 
\mathsf{\Lambda}({\ma{d}}):=  \{ \x \in \Z^2 : d_i \mid L_i(\x), ~ 
d_3\mid Q(\x)\} 
\end{equation}
et
  \begin{equation}
\rho(\bd)=\rho(\bd;L_1,L_2,Q)
  := \# \big(\mathsf{\Lambda}({\ma{d}})\cap  [0,d_1d_2d_3)^2 \big).
  \label{defrho}
\end{equation}

Le r\'esultat principal de notre article est le suivant.
\begin{thm}\lab{tht}
Soit $\varepsilon>0$. Lorsque  $L_1, L_2,Q,\mcal{R}$ satisfont 
(H1)--(H3) et lors\-que $X\geq 2$,
on  a
\begin{equation*}
T(X) = 2C \vol(\mcal{R}) X^2(\log    X)^3 + O\big( X^2 (\log 
X)^{2+\varepsilon} \big), 
\end{equation*}
o\`u la constante implicite dans le $O$ d\'epend 
de $\ve$, des formes lin\'eaires $L_1, L_2,Q$ et de la borne
sup\'erieure de $\mcal{R}$ et o\`u
$$C:=\prod_p\Big(1-\frac{1}{p}\Big)^3\Big(1+\sum_{ \nu  \in\N} 
\frac{\rho(1, p^\nu, 1)
+ \rho(p^\nu  , 1, 1)+ \rho(1  , 1,
p^\nu  )}{p^{2\nu}}\Big).$$
\end{thm}

Ce probl\`eme s'inscrit dans le cadre plus g\'en\'eral de 
l'estimation de la somme
$$  T_F(X)=\sum_{\x\in \Z^2\cap X\mcal{R}}\tau(F(\x))
$$
lorsque $F$ est une forme binaire de degr\'e $4$ de $\Z[\x]$.
Lorsque $F$ est irr\'eductible sur $\R$, Daniel \cite{D99} a montr\'e
$$T_F(X) =4C_F \vol(\mcal{R})X^2 \log X+O(X^2 \log\log X)$$
lorsque $\mcal{R}=\{ \x\in\R^2\,:\, |f(\x)|\leq 1\}$ et
$$C_F=\prod_{p}\Big(1-\frac{1}{p}\Big) \sum_{ \nu  \in\Z_{\geq 
0}}\frac{\rho_F(p^\nu)}{p^{2\nu}},$$
avec $ \rho_F( d ):=\#\{ \x \in [0,d)^2 : d \mid F(\x)\} .$

Nous sommes convaincus que les m\'ethodes d\'evelopp\'ees dans cet 
article seront  utiles pour \'etablir
une estimation asymptotique de $T_F(X)$ lorsque $F$ est une forme 
binaire  de degr\'e $4$ de $\Z[\x]$
qui n'est pas trait\'ee par notre r\'esultat ou celui de Daniel.

\begin{ack}
Une partie de ce travail a \'et\'e r\'ealis\'ee lors de la venue du 
deuxi\`eme auteur \`a l'Universit\'e
Paris 7--Denis Diderot et \`a l'Universit\'e
Paris 6--Pierre et Marie Curie. Que ces institutions soient ici 
chaleureusement remerci\'ees pour leur
accueil et leur soutien financier. Le deuxi\`eme auteur b\'en\'eficie 
du soutien de la bourse  EPSRC
num\'erot\'ee \texttt{EP/E053262/1}.
\end{ack}

\section{M\'ethode et autres r\'esultats}

Pour \'etablir le Th\'eor\`eme \ref{tht}, nous estimons des sommes 
g\'en\'eralisant $S(X)$ en pr\'e\-tant
une attention particuli\`ere \`a l'uniformit\'e par rapport aux 
diff\'erents  param\`etres.
  Pour cela, nous introduisons les
quantit\'es suivantes
\begin{equation}
   \label{eq:Linf}
L_\infty=L_\infty(L_1,L_2, Q):=\max\{\|L_1\|,\|L_2\|,\|Q\|\}
\end{equation}
o\`u $\|{\, \cdot \,}\|$ d\'esigne le maximum  
de la valeur absolue des coefficients de la forme consid\'er\'ee, et
\begin{align}
   \label{eq:rinf}
r_\infty&=r_\infty(\mcal{R}):=\sup_{\x\in\mcal{R}}\max\{|x_1|,|x_2|\},
\\
   \label{eq:r'}
r'&=r'(L_1,L_2,Q,\mcal{R}):=\sup_{\x\in\mcal{R}}\max\{|L_1(\ma{x})|,|L_2(\ma{x})|,\sqrt{|Q(\ma{x})|}\}.
\end{align}

Dans toute la suite, nous consid\'ererons, pour des ensembles $V   $ 
de $\R^3$ d\'efinis comme une
intersection finie d'hyperplans \`a coefficients born\'es, les sommes 
d\'efinies par  
\begin{align*}
S(X;V   )&:=S(X;L_1,L_2,Q,\mcal{R},V)=\sum_{\x\in \Z^2\cap
X\mcal{R}}
\tau(L_1(\x),L_2(\x),Q(\x);V   ),
\end{align*}
avec, lorsque $V   \subseteq [0,1  ]^3$,
$$\tau(L_1(\x),L_2(\x),Q(\x);V   ):=\#\left\{ \ma{d}\in\N^3 :
\!\!\begin{array}{l}  d_i\mid L_i(\x),\quad d_3\mid Q(\x) \\ {}\\ 
\Big(\frac{\log d_1}{\log
r'X},\frac{\log d_2}{\log r'X},\frac{\log d_3}{2\log r'X}\Big)\!\!\in 
V   \end{array}\right\}.$$
La d\'ependance de $\tau(L_1(\x),L_2(\x),Q(\x);V   )$ \`a $X$ et 
$\mcal{R}$ sera toujours omise.
Lorsque $V=[0,1  ]^3$, nous avons 
$S(X,V)=S(X).$

\begin{thm}\lab{main0}
Soit $\varepsilon>0$. Lorsque  $L_1, L_2,Q,\mcal{R}$ satisfont (H1)--(H3),
  $V   $ est un sous-ensemble de $[0,1]^3$ d\'efini comme une
intersection finie d'hyperplans \`a coefficients born\'es et
  $r'X^{1-\varepsilon }\geq 1$, on a
\begin{equation}
  \begin{split}
S (X;V)
=& 
2 \vol(\mcal{R})\vol(V   )X^2(\log r'X)^3 \prod_{p}\sigma_p \\
& + O\Big( L_\infty^{ \varepsilon} (r_\infty r'+r_\infty^2) X^2 (\log X)^{2}
\log\log X\Big),
 \end{split}
  \label{estS(X)}
\end{equation}
o\`u
\begin{equation}
    \label{defsigma}
\sigma_p :=\Big(1-\frac{1}{p}\Big)^3
\sum_{\mnu\in\Z_{\geq 0}^3} \frac{\rho
(p^{\nu_1},p^{\nu_2},p^{\nu_3})}{p^{2\nu_1+2\nu_2+2\nu_3}}.
\end{equation}
\end{thm}

L'ensemble \eqref{eq:farm}  
n'est pas un r\'eseau.
L'\'etape essentielle de la preuve de cette estimation est de se ramener \`a un
comptage sur des r\'eseaux de $\Z^2$ en s'inspirant de l'important article
\cite{D99}.

Le Th\'eor\`eme \ref{main0} nous permet d'estimer des sommes plus 
g\'en\'erales de
la forme
\begin{align*}
   S(X,\ma{d},\ma{D};V)  =& 
\,\,S(X,\ma{d},\ma{D};L_1,L_2,Q,\mcal{R},V)\\   :=&\sum_{\x\in
\mathsf{\Lambda}({\ma{D}})\cap 
X\mcal{R}}\tau\Big(\frac{L_1(\x)}{d_1},\frac{L_2(\x)}{d_2},\frac{Q(\x)
}{d_3};V\Big) ,
\end{align*}
lorsque $\ma{d},\ma{D}\in\N^3$ tels que $d_i\mid D_i$.

Consid\'erons
\begin{equation}
    \label{defsigmad}
\sigma_p(\ma{d},\ma{D}) :=\Big(1-\frac{1}{p}\Big)^3
\sum_{\mnu\in\Z_{\geq 0}^3}\frac{\rho
(p^{N_1},p^{N_2},p^{N_3})
}{p^{2N_1+2N_2+2N_3}},
\end{equation}
avec \begin{equation}\label{defNi}\lambda_i=v_p(d_i),\quad \mu_i=v_p(D_i),\quad
N_i =\max\{ \mu_i,\nu_i+\lambda_i  \} .
\end{equation}
Soit $\delta(\ma{D})$ le plus grand entier $\delta$ tel que
\begin{equation}\label{defdelta}
\mathsf{\Lambda}({\ma{D}})\subseteq\{ \x\in\Z^2\,:\quad \delta\mid
(x_1,x_2)\} ,\end{equation}
o\`u $(x_1,x_2)$ d\'esigne, ici et dans toute la suite, le pgcd des 
entiers $x_1$ et $x_2$.
Lorsque  $L_1, L_2,Q  $ des formes satisfont (H2), nous \'ecrirons
$\ell_1,\ell_2,q$ des entiers  et $L_1^*, L_2^*,Q^*  $ des formes 
primitives tels que
\begin{equation}\label{ecritprim}L_1=\ell_1L_1^*,\quad 
L_2=\ell_2L_2^*,\quad Q  =qQ^*,
\end{equation}
et utilisons la notation 
\begin{equation}\label{defD'}
\ma{D'}:=\Big(\frac{D_1}{(D_1,\ell_1)},\frac{D_2}{(D_2,\ell_2)},
\frac{D_3}{(D_3,q)}\Big).
\end{equation}
Nous emploierons syst\'ematiquement la notation
$$D:=D_1D_2D_3 $$ et 
\begin{equation}\label{defdisc}\begin{split}
\Delta_{12}&:= \Res (L_1,L_2)=a_1b_2-a_2b_1,\\
\Delta_{i3}&:= \Res 
(L_i,Q)=a_3b_i^2+b_3a_i^2-c_3a_ib_i=Q(-b_i,a_i).\end{split} 
\end{equation}
  Comme $Q$ est irr\'eductible,
nous avons $\Delta_{i3}\neq 0$. De plus, comme $L_1$ et $L_2$ ne sont 
pas proportionnelles, on a
$\Delta_{12}\neq 0$.
Nous notons aussi
\begin{equation}\label{def1aDD}
a(\ma{D},\mDelta):=
(D_1,\Delta_{12}) (D_2,\Delta_{12})(D_3,\Delta (\Delta_{13},\Delta_{23})), 
\end{equation} 
avec    $\mDelta  :=(\Delta, \Delta_{12},
\Delta_{13}, \Delta_{23})$ et $\Delta$ introduit en \eqref{defDelta}.
Enfin, nous posons
\begin{equation}
a'(\ma{D},\mDelta):=
(D_1',\Delta'_{12}) (D_2',\Delta'_{12})(D_3,\Delta' (\Delta'_{13},\Delta'_{23})),
\label{defaDD}
\end{equation}
o\`u $\mDelta'  :=(\Delta', \Delta'_{12},
\Delta'_{13}, \Delta'_{23})= (\Delta/q^2, \Delta_{12}/\ell_1\ell_2,
\Delta_{13}/\ell_1^2q, \Delta_{23}/\ell_2^2q).$ Autrement $a'$ prend la valeur
$a$ apr\`es avoir rendu les formes $L_1, L_2,Q  $  primitives.
\`A partir du Th\'eor\`eme \ref{main0}, nous d\'emontrerons le 
r\'esultat suivant.

\begin{thm}\lab{maind}
Soit $\varepsilon>0$. Lorsque  $L_1, L_2,Q,\mcal{R}$ satisfont (H1)--(H3),
  $V   $ est un sous-ensemble de $[0,1]^3$ d\'efini comme une
intersection finie d'hyperplans \`a coefficients born\'es,  $\ma{d}$ 
et $\ma{D}$ tels que
$d_i\mid D_i$ et
$r' X^{1-\varepsilon }\geq 1$, on a
\begin{align*}
S (X,\ma{d},\ma{D};V)
=& 2 \vol(\mcal{R})\vol(V)
  X^2(\log r'X)^3 \prod_{p}\sigma_p(\ma{d},\ma{D})
  \\& +
O\Big(   \frac{(DL_\infty)^{\varepsilon}}{\delta(\ma{D})}
a'(\ma{D},\mDelta)(r_\infty
r'  +r_\infty^2)    X^2 (\log X)^{2} \log\log X\Big).  
\end{align*}
   De plus, on a  
$\prod_p \sigma_p (\ma{d},\ma{D})  \ll L_\infty^{ \varepsilon}D^{\varepsilon
}a'(\ma{D},\mDelta).$
\end{thm}

La somme $S (X,\ma{d},\ma{D};V)$ est directement reli\'ee au cardinal   
des points \`a coor\-donn\'ees
enti\`eres sur la vari\'et\'e affine d'\'equation
\begin{equation}\label{defvar}
L_i(\x)=d_is_it_i,\qquad Q(\x)=d_3s_3t_3.
\end{equation}
La constante attendue dans le terme principal de  $S 
(X,\ma{d},\ma{D};V)$ peut \^etre  
interpr\'et\'ee comme un produit convenable de densit\'es locales 
avec  $\omega_\infty(\mcal{R},V)$ la
densit\'e archi\-m\'edienne associ\'ee  \`a la vari\'et\'e d\'efinie par
  \eqref{defvar}.
Cette quantit\'e sera explicitement d\'efinie en \eqref{defomega}.
  De la m\^eme mani\`ere que dans \cite{4linear} pour les sommes
  introduites en~\eqref{sumR}, nous montrons que le terme principal est
bien celui qui  est attendu.

Lorsque $\mal=(\lambda_1,\lambda_2, \lambda_3)\in\Z_{\geq 0}^3$, 
$\mmu=(\mu_1,\mu_2, \mu_3)\in\Z_{\geq
0}^3$ et $p$ un nombre premier, nous consid\'erons
$$
N_{\mal,\mmu}(p^n):=\#\left\{(\x,\ma{s},\ma{t})\in
(\Z/p^n\Z)^{8}: \begin{array}{l}
L_i(\x)\equiv p^{\lambda_i}s_it_i \smod{p^n}\\
Q(\x)\equiv p^{\lambda_3}s_3t_3 \smod{p^n}\\
p^{\mu_i} \mid L_i(\x),\, p^{\mu_3} \mid Q(\x)
\end{array}
\right\},
$$ et
\begin{equation*} 
\omega_{\mal,\mmu}(p):=\lim_{n\rightarrow 
\infty}p^{-5n-\la_1-\la_2-\la_3}N_{\mal,\mmu}(p^n).
\end{equation*}
Cette quantit\'e correspond \`a la densit\'e  $p$-adique associ\'ee 
\`a la vari\'et\'e d\'efinie par~\eqref{defvar}.

\begin{thm}\lab{constanteSd} Soient $L_1, L_2,Q,\mcal{R}$ satisfont 
(H1)--(H3) et  $\ma{d}$, $\ma{D}$  
tels que
$d_i\mid D_i$.  
On a $ \omega_\infty(\mcal{R},V)=2 \vol(\mcal{R})\vol(V)$ et
$\omega_{\mal,\mmu}(p)=\sigma_p(\ma{d},\ma{D})$,
pour tout $p$.
\end{thm}

Le  Th\'eor\`eme \ref{maind} permet facilement de d\'emontrer 
d'autres r\'esultats dont nous aurons
besoin dans un travail ult\'erieur concernant la conjecture de Manin 
et qui ont \'et\'e \`a l'origine de
ce travail.

Le premier concerne l'estimation des sommes
\begin{align*}
   S^*(X,\ma{d},\ma{D}; V)  =& \,\,S^*(X,\ma{d},\ma{D};L_1,L_2,Q,\mcal{R},V)\\
:=&\sum_{\colt{\x\in
\mathsf{\Lambda}({\ma{D}})\cap
X\mcal{R}}{(x_1,x_2)=1}}\tau\Big(\frac{L_1(\x)}{d_1},\frac{L_2(\x)}{d_2},\frac{Q(\x)
}{d_3};V\Big) ,
\end{align*}
lorsque $\ma{d},\ma{D}\in\N^3$ tels que $d_i\mid D_i$.
Pour exprimer notre r\'esultat, nous introduisons
\begin{equation} \begin{split}
\rho^*(\bd)&\,\,=\rho^*(\bd;L_1,L_2,Q) \\&
  :=  \# \big\{\x\in\mathsf{\Lambda}({\ma{d}})\cap 
[0,d_1d_2d_3)^2\,:\, (x_1,x_2,d_1d_2d_3)=1\big\}.
\end{split}\label{defrho*}\end{equation}
Lorsque $v_p(D)\geq 1$, nous d\'efinissons
\begin{equation} \label{defsigma*d}
\sigma^*_p(\ma{d},\ma{D}) 
:=\Big(1-\frac{1}{p}\Big)^3\sum_{\mnu\in\Z_{\geq 0}^3}\frac{\rho^*
(p^{N_1},p^{N_2},p^{N_3})}{p^{2N_1+2N_2+2N_3}},
\end{equation}
o\`u les $N_i$ sont d\'efinis par \eqref{defNi} alors que lorsque $v_p(D)=0$
\begin{equation} \label{defsigma*d0}
\sigma^*_p(\ma{d},\ma{D}) :=\Big(1-\frac{1}{p}\Big)^3
\Big\{ 1-\frac{1}{p^2} +
\sum_{\colt{\mnu\in\Z_{\geq 0}^3}{\nu_1+\nu_2+\nu_3\geq 1}}\frac{\rho^*
(p^{\nu_1},p^{\nu_2},p^{\nu_3})}{p^{2\nu_1+2\nu_2+2\nu_3}}\Big\}.
\end{equation}

\begin{Cor}\lab{corS*}
Soit $\varepsilon>0$. Lorsque  $L_1, L_2,Q,\mcal{R}$ satisfont 
(H1)--(H3),  $V   $ est un sous-ensemble de $[0,1]^3$ d\'efini comme 
une
intersection finie d'hyperplans \`a coefficients born\'es,  $\ma{d}$ 
et $\ma{D}$ tels que
$d_i\mid D_i$ et
$r' X^{1-\varepsilon }\geq 1$, on a
\begin{align*}
S^*(X,\ma{d},\ma{D};V)
=& 2 \vol(\mcal{R})\vol(V)
  X^2(\log r'X)^3 \prod_{p}\sigma^*_p(\ma{d},\ma{D})
  \\& +
O\Big(  {(DL_\infty)^{\varepsilon}}a'(\ma{D},\mDelta) (r_\infty r' 
+r_\infty^2)    X^2 (\log X)^{2+\varepsilon}\Big).
  \end{align*}
\end{Cor}

Enfin, but ultime et motivation principale de cet article, nous 
estimations la somme
$$
   T_g^*(X;V)
=\sum_{\colt{\x\in \Z^2\cap
X\mcal{R}}{(x_1,x_2)=1}}g(L_1(\x), L_2(\x), Q(\x);V).
$$
o\`u $g=\tau*h$ est une fonction arithm\'etique multiplicative proche 
de $\tau$ au sens de la
convolution, $V\subseteq [0,1]^3$  et
$$g(L_1(\x), L_2(\x), Q(\x);V):=\sum_{\tolt{d\mid 
L_1(\x)L_2(\x)Q(\x)}{d_i=(d, L_i(\x)),\, d_3=(d,
Q(\x))}{
\big(\frac{\log d_1}{\log r'X},\frac{\log d_2}{\log r'X},\frac{\log 
d_3}{2\log r'X}\big)\in
V}}(1*h)(d ),
$$
de sorte que
$g(L_1(\x), L_2(\x), Q(\x);[0,1]^3)=g(L_1(\x)  L_2(\x)  Q(\x)).$

Nous introduisons le cardinal
  $ {\rho}^\dagger_p( \nu_1 ,\nu_2,\nu_3)$ d\'efini 
lorsque $\nu=\nu_1+\nu_2+\nu_3\geq 0$  
par
$${\rho}^\dagger_p( \nu_1 ,\nu_2,\nu_3):=
\#\left\{ \x\in \Z^2\cap [0,p^{\nu +1})^2\,:\begin{array}{l} p^{\nu_i}\parallel
L_i(\x) ,\\p^{\nu_3}\parallel
Q(\x),\\   (p,x_1,x_2)=1\end{array}\right\}.$$
Ici, pour $\nu$ et $n$ la relation $p^{\nu}\parallel  
n$ signifie que $v_p(n)=\nu$.  
Nous introduisons aussi
  la densit\'e $\overline{\rho}^\dagger_p( \nu_1 ,\nu_2,\nu_3)$ d\'efinie par
$$\overline{\rho}^\dagger_p(\nu_1,\nu_2,\nu_3):=\frac{1}{p^{2(\nu +1)}}
  {\rho}^\dagger_p( \nu_1 ,\nu_2,\nu_3).$$
Cette quantit\'e s'exprime en fonction de la fonction $\rho^*$ 
gr\^ace \`a la formule
\begin{equation}\label{ovrhoP}\overline{\rho}^\dagger_p(\nu_1,\nu_2,\nu_3)=
\sum_{
  \msigma\in\{0,1\}^3} (-1)^{\sigma_1+\sigma_2+\sigma_3}
\frac{\rho^*(p^{\nu_1+\sigma_1},p^{\nu_2+\sigma_2},p^{\nu_3+\sigma_3})}{p^{2(\nu_1+ 
\nu_2+ \nu_3+
\sigma_1+ \sigma_2+ \sigma_3)}},
\end{equation}
avec la convention $\rho^*(1,1,1)=1-1/p^2$.
\goodbreak

Nous consid\'erons les fonctions $g$ telles que la fonction 
$h=g*\mu*\mu$ satisfasse
\begin{equation}\label{conditiong}\sum_{d\in\N}\frac{|h(d)|}{d^{1/2-\eta_0}}\ll 
1
\end{equation}
pour une constante $\eta_0>0$.

\begin{Cor}\lab{corT*}
Soient $\varepsilon>0$,  $L_1, L_2,Q,\mcal{R}$ satisfaisant 
(H1)--(H3),  $V   $ est un  sous-ensemble de
$[0,1]^3$ d\'efini comme une
intersection finie d'hyperplans \`a coefficients born\'es    et $g$ 
une fonction
multiplicative satisfaisant \eqref{conditiong} pour un $\eta_0>0$ 
fix\'e.  Lorsque
$X\geq 2$,  on  a
\begin{equation*}
T_g^*(X;V) = 2C^* \vol(\mcal{R})\vol(V) X^2(\log    X)^3 + O\big( 
X^2 (\log X)^{2+\varepsilon}
\big), 
\end{equation*}
o\`u la constante implicite dans le $O$ d\'epend  des formes 
lin\'eaires $L_1, L_2,Q$, de la borne 
sup\'erieure de $\mcal{R}$ et des constante $\eta_0$ et $\varepsilon$  et o\`u
\begin{equation}\label{defC*}C^*:=\prod_p\Big(1-\frac{1}{p}\Big)^3\Big\{ 
\sum_{ \mnu\in\Z_{\geq
0}^3} g (p^{\nu_1+\nu_2+\nu_3})
  \overline{\rho}^\dagger_p(\nu_1,\nu_2,\nu_3)\Big\}.\end{equation}
\end{Cor}

Le Corollaire \ref{corT*} permet d'avoir une version du Th\'eor\`eme 
\ref{tht} avec une somme
restreinte aux couples \`a coordonn\'ees premi\`eres entre elles en 
pre\-nant~$g=\tau$ et $V=[0,1]^3$.

Le r\'esultat suivant permet d'expliciter comment on peut faire 
appara\^\i tre des conditions portant sur
$\log \max\{|x_1|,|x_2|\}$ dans le volume $V$. Soit la somme
$$
   T_g'(X;V')
=\sum_{\colt{\x\in \Z^2\cap
X\mcal{R}}{(x_1,x_2)=1}}\frac{g'(L_1(\x), L_2(\x), 
Q(\x);V')}{\max\{|x_1|,|x_2|\}^2}.
$$
o\`u $V'\subseteq [0,1]^4$ et
$$g'(L_1(\x), L_2(\x), Q(\x);V'):=
\!\!\!\!\!\!\!\!\!\!\!\!\!\!\!\!\!\!\!\!\!\!
\sum_{\tolt{d\mid L_1(\x)L_2(\x)Q(\x)}{d_i=(d, L_i(\x)),\, d_3=(d,
Q(\x))}{
\big(\frac{\log d_1}{\log  Y},\frac{\log d_2}{\log  Y},\frac{\log 
d_3}{2\log  Y},\frac{\log
\max\{|x_1|,|x_2|\}}{ \log  Y}\big)\in V'}}(1*h)(d ).
$$
La d\'efinition de cette fonction d\'epend de $Y$ qui d\'ependra de $X$.
Nous notons aussi
  $$V_0'(u):=\{ \ma{t}\in[0,1]^4\,:\, t_i\leq  t_4\leq u\quad (1\leq 
i\leq 3)\}.$$
Nous d\'eduisons 
du Corollaire \ref{corT*}  l'estimation de $T'_g(X;V') $ suivante.
\begin{Cor}\lab{corT'}
Soient $\varepsilon>0$,  $L_1, L_2,Q,\mcal{R}$ satisfaisant 
(H1)--(H3) tel que $r'\asymp 1$,  $V'$ est
un  sous-ensemble de
$[0,1]^4$ d\'efini comme une
intersection finie d'hyperplans \`a coefficients born\'es    et $g$ 
une fonction
multiplicative satis\-faisant~\eqref{conditiong} pour un $\eta_0>0$ 
fix\'e.  Lorsque
$2\leq X\leq Y\leq X^{1/\varepsilon}$,  on  a
\begin{equation*}
T_g'(X;V') = 4 C^* \vol(\mcal{R})\vol(V'\cap V_0'(\log X/\log Y)) 
(\log    Y)^4 + O\big((\log X)^{3+\varepsilon}\big),  
\end{equation*}
o\`u la constante implicite dans le $O$ d\'epend des formes 
lin\'eaires $L_1, L_2,Q$, de la borne 
sup\'erieure de $\mcal{R}$ et des constantes $\eta_0$  et $\varepsilon$ et o\`u
$C^*$  est  d\'efinie en~\eqref{defC*}.
\end{Cor}

Ce corollaire est un des points cl\'es de \cite{dp4smooth} dans 
lequel les auteurs \'etablissent une
formule asymptotique pour le nombre de points rationnels de hauteur 
contr\^ol\'ee sur une surface 
del Pezzo de degr\'e 4 non-singuli\`ere. 
Pour cette application, il est 
important d'avoir le choix de
l'ensemble $V$.

\smallskip\goodbreak
D\'etaillons maintenant le contenu de l'article. Dans la   section 
\ref{sectionrho}, nous \'enon\c cons et
d\'emontrons les propri\'et\'es utiles de $\rho^*$ et $\rho$. Cela 
permet alors de montrer la convergence
du produit eul\'erien intervenant dans l'expression du terme 
principal du Th\'eor\`eme \ref{main0}. Dans
la section \ref{sectionconstante}, une interpr\'etation 
g\'eom\'etrique de la constante du Th\'eor\`eme
\ref{maind} est donn\'ee pour d\'emontrer le Th\'eor\`eme 
\ref{constanteSd}. La section \ref{sectionmain0}
est le c\oe ur de l'article qui contient la d\'emonstration du 
Th\'eor\`eme \ref{main0}. La
d\'emonstration du  Th\'eor\`eme \ref{maind}, qui est d\'etaill\'ee 
dans la section \ref{sectionmaind},
repose sur une manipulation d\'elicate qui permet de se ramener
\`a des sommations sur des r\'eseaux puis \`a des sommes trait\'ees 
par le Th\'eor\`eme~\ref{main0}
apr\`es changement de variables. La fin de la section est consacr\'ee 
au calcul explicite de la constante
intervenant dans le terme principal. On trouvera la d\'emonstration 
du Th\'eor\`eme \ref{tht} \`a la
section \ref{sectiontht}. Celle-ci consiste \`a exprimer la somme 
$T(X)$ en fonction des sommes
estim\'ees au Th\'eor\`eme \ref{maind}. La section \ref{sectionCor} 
contient la d\'emonstration des trois
corollaires.

  \goodbreak

\section{Propri\'et\'es des fonctions $\rho^*$ et $\rho$}\lab{sectionrho}

Nous rappelons les d\'efinitions \eqref{defrho} et \eqref{defrho*}.
Nous rassemblons ici les informations dont nous aurons besoin 
concernant la fonction $\rho$ qui a \'et\'e
notamment \'etudi\'ee dans \cite{D99}, \cite{M06} et  \cite{M08}. 
Dans la plupart des cas trait\'es dans
la litt\'erature, il est suppos\'e que les formes sont primitives 
c'est-\`a-dire que les coefficients de
chaque forme lin\'eaire sont premiers entre eux. Nous attacherons une 
importance toute particuli\`ere \`a
l'uniformit\'e des relations par rapport aux coefficients des formes 
$L_1, L_2,Q  $.

Le lemme suivant permet de se ramener \`a des formes primitives.

\begin{lem}\label{nonprimitive} Soient  $L_1, L_2,Q  $ des formes 
satisfaisant (H2)
et la notation \eqref{ecritprim}.
Alors, lorsque $\bd\in\N^3$ et $d_i'= d_i/(d_i,\ell_i)$,  $d_3'= 
d_3/(d_3,q)$, on a
$$\frac{\rho(\bd;L_1,L_2,Q)}{(d_1 d_2 d_3)^2}
=\frac{\rho(\bd';L_1^*,L_2^*,Q^*)}{(d_1' d_2' d_3')^2} .$$
\end{lem}

La d\'emonstration est imm\'ediate, nous n'indiquons pas les d\'etails.
Les fonctions $\rho$ et $\rho^*$ sont multiplicatives. Nous pouvons 
donc nous contenter de les \'etudier
sur des triplets de puissance de nombre premier.
Nous rappelons les notations \eqref{defdisc}.
  \goodbreak

\begin{lem}\label{rho*premier}
   Soient  $L_1, L_2,Q  $ des formes primitives satisfaisant (H2). \par 1) Alors
pour tout $p$ premier et $\nu\in\Z_{\geq 0}$, on a
$$\rho^*(p^\nu,1,1)=\rho^*(1,p^\nu,1)=\varphi(p^\nu).$$
\par 2)
Lorsque $p\nmid 2 {\rm disc}( Q)$  et $\nu\in\Z_{\geq 0}$, alors on a
$$\rho^*(1,1,p^\nu)=\varphi(p^\nu)\Big\{ 1+\Big(\frac{{\rm disc}( 
Q)}{p}\Big)\Big\} .$$
De plus si $p$ est un facteur impair de  ${\rm disc}( Q)$, on a
$$\rho^*(1,1,p^\nu)\leq 2\varphi(p^\nu) p^{\min\{[v_p({\rm disc}( 
Q))/2],[\nu/2]\}} .$$
Nous avons aussi
$$\rho^*(1,1,2^\nu)\leq  2^{\nu+2} .$$

\par 3) Pour tout $p$, on a
$$\rho(p^{\nu_1},p^{\nu_2},p^{\nu_3})
=\!\!\!\!\!\!\!\!\!\!\!\!\!\!\!\!\!\!\sum_{0\leq k\leq \max\{\nu_1,\nu_2,\lceil
\nu_3/2\rceil \}}\!\!\!\!\!\!\!\!\!
\rho^*\big(p^{\max\{\nu_1-k,0\}},p^{\max\{\nu_2-k,0\}},p^{\max\{\nu_3-2k,0\}}\big)
p^{m_k}$$
avec $m_k:=2(\min\{\nu_1,k\}+\min\{\nu_2,k\}+\min\{\nu_3,2k\}-k)$.
\par 4) Pour tout $p$ premier, nous avons
$$\rho^*(p^{\nu_1},p^{\nu_2},p^{\nu_3})=0$$
lorsqu'il existe $1\leq i<j\leq 3$ tels que
$v_p(\Delta_{ij})<\min\{\nu_i,\nu_j\}$.  En particulier, cette 
\'egalit\'e est satisfaite
lorsque
$p\nmid
\Delta_{13}\Delta_{23}\Delta_{12}$ et    $\#\{i\,:\, \nu_i\geq 1\}\geq 2.$
\par 5)    Pour tout $p$ premier et $\nu_3\leq\max\{\nu_1,\nu_2\}$, nous avons
$$\rho^*(p^{\nu_1},p^{\nu_2},p^{\nu_3})\leq   \varphi (p^{\nu_1+\nu_2+\nu_3})
p^{ \nu_3+\min\{\nu_1,\nu_2, v_p(\Delta_{12})\}}.$$
    Pour tout $p$ premier et  $ \max\{\nu_1,\nu_2\}<\nu_3$, nous avons
$$\rho^*(p^{\nu_1},p^{\nu_2},p^{\nu_3})\leq   8\varphi (p^{\nu_1+\nu_2+\nu_3})
p^{ \nu_1+ \nu_2} p^{\min\{[v_p({\rm disc}( Q))/2],[\nu_3/2]\}} . $$
\end{lem}

\begin{proof}
Ces r\'esultats sont essentiellement d\'emontr\'es dans  \cite{M08}, 
mais la n\'ecessit\'e de formules
uniformes en fonction  des formes $L_1$, $L_2$ et $Q$  demande de 
reprendre toutes les d\'emonstrations.

Le  point (1) et la premi\`ere partie du point (2) sont clairs 
puisque les formes sont suppos\'ees
primitives. Le point (3) est
\'etabli dans
\cite{M06}.

  Nous indiquons les d\'etails de la preuve
de la majoration de
$\rho^*(1,1,p^\nu)$ \'enonc\'ee au point (2) lorsque
$p$ est un facteur impair de  ${\rm disc}( Q)$. On peut supposer que 
parmi $a_3$ et $b_3$ au moins un des
deux est premier \`a $p$. Supposons ainsi que $(a_3,p)=1$. Le 
changement de variables
$y_1=x_1+c_3x_2/(2a_3)$ et $y_2=x_2/2a_3$ permet de montrer que
$$\rho^*(1,1,p^\nu)=\#\{ (y_1,y_2)\in(\Z/p^\nu\Z)^2\,:\, 
y_1^2-p^k\delta y_2^2\equiv 0\mod{p^\nu},\,
p\nmid (y_1,y_2)\},$$ o\`u ${\rm disc}( Q)=p^k\delta$ et $(\delta,p)=1$.
Si $\nu\leq k$, ces conditions se tranforment en $(p,y_2)=1$ et 
$p^{\lceil \nu/2\rceil}\mid y_1$ ce qui
fournit
$$
\rho^*(1,1,p^\nu)=\varphi(p^\nu)p^{[\nu/2]} .$$
Si $\nu>k$, alors $\rho^*(1,1,p^\nu)=0$ si $k$ est impair et si $k$ 
est pair alors
$y_1=p^{k/2}y_1'$ de sorte  que
\begin{align*}\rho^*(1,1,p^\nu)&=p^{3k/2}
\#\{ (y_1,y_2)\in((\Z/p^{\nu-k}\Z)^*)^2\,:\,     y_1^2- \delta 
y_2^2\equiv 0\mod{p^{\nu-k}} \}
\\&\leq 2\varphi(p^\nu)p^{k/2}=2\varphi(p^\nu)p^{[k/2]},\end{align*} 
ce qui fournit bien l'in\'egalit\'e
annonc\'ee.

Consid\'erons maintenant le cas de $\rho^*(1,1,2^\nu)$.
Lorsque $2^{\nu} \mid Q(\x)$ et $2\nmid (x_1,x_2)$, au moins une des 
deux coordonn\'ees $x_i$ est
impaire. Appelons $x_j$ l'autre coordonn\'ee. Alors $x_j/x_i$ est 
solution d'un polyn\^ome de degr\'e 2
mo\-dulo $2^n$ qui a donc au plus $4$ solutions modulo $2^n$. 
L'in\'egalit\'e recherch\'ee en d\'ecoule.

Pour d\'emontrer le point (4), remarquons que lorsqu'il y a une 
solution compt\'ee par
$\rho^*(p^{\nu_1},p^{\nu_2},p^{\nu_3})$ nous avons lorsque $1\leq i<j\leq 3$
$$p^{\min\{\nu_i,\nu_j\}}\mid \Delta_{ij}(x_1,x_2)=\Delta_{ij},$$
ce qui fournit le r\'esultat annonc\'e.

Le point (4) permet de restreindre au cas o\`u 
$\min\{\nu_1,\nu_2\}\leq v_p(\Delta_{12})$,
$ \min\{\nu_1,\nu_3\}\leq v_p(\Delta_{13})$,
$\min\{\nu_2,\nu_3\}\leq v_p(\Delta_{23})  $ pour d\'emontrer le 
point (5). Supposons par exemple que
$\nu_1\geq \max\{\nu_2,\nu_3\}$. Nous avons
$$ \rho^*(p^{\nu_1},p^{\nu_2},p^{\nu_3})\leq 
p^{2\nu_2+2\nu_3}\rho^*(p^{\nu_1},1,1)$$ d'o\`u le
r\'esultat. Le cas $\nu_2\geq \max\{\nu_1,\nu_3\}$ est identique.
  Lorsque   $ \nu_3>\max\{\nu_1,\nu_2\} $, nous avons
\begin{align*}
\rho^*(p^{\nu_1},p^{\nu_2},p^{\nu_3})&\leq 
p^{2\nu_1+2\nu_2}\rho^*(1,1,p^{\nu_3})
\end{align*}
ce qui fournit le r\'esultat gr\^ace au point (2).
\end{proof}
De ces informations sur $\rho^*$, nous d\'eduisons les propri\'et\'es 
suivantes.

\begin{lem}\label{rhopremier}
   Soient  $L_1, L_2,Q  $ des formes primitives satisfaisant (H2). \par 1) Alors
pour tout $p$ premier et $\nu\in\Z_{\geq 0}$ on a
$$\rho(p^\nu,1,1)=\rho(1,p^\nu,1)=p^\nu.$$
\par 2) Lorsque $p\nmid 2 {\rm disc}( Q)$  et $\nu\in\Z_{\geq 0}$, alors on a
$$\rho(1,1,p^\nu)=\varphi(p^\nu)\Big\{ 1+\Big(\frac{{\rm disc}( 
Q)}{p}\Big)\Big\}\lceil \nu/2\rceil
+p^{2(\nu-\lceil
\nu/2\rceil)} \leq (\nu+1)p^\nu ,$$
et pour tout facteur $p$ impair de $   {\rm disc}( Q)$, on a
$$\rho(1,1,p^\nu)\ll (\nu +1)p^{\nu+\min\{[v_p({\rm disc}( Q))/2],[\nu/2]\}}.$$
Enfin, nous avons
$$\rho(1,1,2^\nu)\ll (\nu +1)2^\nu.$$
\par 3) Lorsque $\max\{\nu_1,\nu_2\}\leq \lceil\dm\nu_3\rceil$ et $p$ 
impair, on a
$$\rho(p^{\nu_1},p^{\nu_2},p^{\nu_3})\ll (\nu_3+1)p^{2(\nu_1+\nu_2)+ \nu_3 }
p^{\min\{[v_p({\rm disc}( Q))/2],[\nu_3/2]\}}.$$
De plus lorsque $\max\{\nu_1,\nu_2\}= 1=\nu_3$ et $p\nmid 
\Delta_{12}\Delta_{13}\Delta_{23}$, on a
$$\rho(p^{\nu_1},p^{\nu_2},p^{\nu_3})\ll  p^{2(\nu_1+\nu_2) }.$$
Enfin lorsque $\max\{\nu_1,\nu_2\}\leq \lceil\dm\nu_3\rceil$
$$\rho(2^{\nu_1},2^{\nu_2},2^{\nu_3})\ll  2^{2 \nu_1+2\nu_2+\nu_3 }.$$
\par  4)
Lorsque $\lceil \dm\nu_3\rceil \leq \min\{\nu_1,\nu_2\}$ et tout $p$ 
premier, on a
$$\rho(p^{\nu_1},p^{\nu_2},p^{\nu_3})\ll
  p^{ \nu_1+\nu_2   +2\nu_3+\min\{\nu_1,\nu_2, v_p(\Delta_{12})\}
}.$$
\par  5)
Lorsque $ \nu_j\leq \lceil\dm\nu_3\rceil\leq \nu_3\leq  \nu_i$ avec 
$\{i,j\}=\{1,2\}$ et tout  $p $
premier, on a
$$\rho(p^{\nu_1},p^{\nu_2},p^{\nu_3})\ll    p^{ 2\nu_j  +\nu_i+ 
\nu_3+[\nu_3 /2 ]+\min\{ \lceil
\nu_3/2\rceil, \lceil v_p(\Delta_{i3})/2\rceil\} }.$$
Lorsque $ \nu_j\leq \lceil\dm\nu_3\rceil\leq   \nu_i<\nu_3$ avec 
$\{i,j\}=\{1,2\}$ et tout  $p $ premier,
on a
$$\rho(p^{\nu_1},p^{\nu_2},p^{\nu_3})\ll
\nu_3 p^{ 2\nu_j  +\nu_i+ \nu_3+  [\nu_3/2] }\big(p^{\min\{ \lceil 
\nu_3/2\rceil, \lceil
v_p(\Delta_{i3})/2\rceil\}}
+p^{r_p}
\big),$$ avec $$r_p:=\min\{
  \nu_i-[\nu_3/2],\lceil 
v_p(\Delta_{i3})/2\rceil\}+\min\{[\nu_3/2],[v_p({\rm disc}( 
Q))/2]\}.$$
\end{lem}

Les constantes implicites dans les majorations du lemme sont 
ind\'epen\-dantes des coefficients des formes
binaires $L_1$, $L_2$ et $Q$.

\begin{proof}Le premier point est trivial puisque les $L_i$ sont primitives.

Pour la deuxi\`eme assertion, nous avons
$$\rho(1,1,p^{\nu }) =\sum_{0\leq k\leq  \lceil
\nu  /2\rceil  }
\rho^*\big(1,1,p^{\max\{\nu -2k,0\}}\big)
p^{2\min\{\nu ,2k\}-2k}.$$
Le cas $p\nmid 2 {\rm disc}( Q)$  est une cons\'equence directe du 
point (2) du  Lemme
\ref{rho*premier}.
Supposons~$p$ un facteur premier  impair de ${\rm disc}( Q)$. 
D'apr\`es le point (2) du Lemme
\ref{rho*premier}, nous avons
\begin{align*}\rho^*\big(1,1,p^{\max\{\nu -2k,0\}}\big)
p^{2\min\{\nu ,2k\}-2k}&\leq  2p^\nu(1-1/p)p^{\min\{ 
[\nu/2-k],[v_p({\rm disc}( Q))/2]\}}\\
&\leq  2p^\nu(1-1/p)p^{\min\{ [\nu/2 ],[v_p({\rm disc}( Q))/2]\}},
\end{align*}
ce qui fournit
$$\rho(1,1,p^{\nu })\leq (\nu+2) p^\nu p^{\min\{ [\nu/2 ],[v_p({\rm 
disc}( Q))/2]\}}.$$
La derni\`ere in\'egalit\'e du point (2) se d\'eduit directement de 
celle qui lui correspond  au Lemme
\ref{rho*premier}.

  D\'emontrons le point (3).
  En ignorant les conditions $p^{\nu_i}\mid L_i(\x) $ dans la d\'efinition de
$\rho(p^{\nu_1},p^{\nu_2},p^{\nu_3})$, nous obtenons
$$\rho(p^{\nu_1},p^{\nu_2},p^{\nu_3})\leq p^{2(\nu_1+\nu_2) 
}\rho(1,1,p^{\nu_3})$$
ce qui gr\^ace au (2) permet d'obtenir le r\'esultat recherch\'e pour 
$\max\{\nu_1,\nu_2\}\leq
\lceil\dm\nu_3\rceil$ et $p$ impair. Lorsque 
$\max\{\nu_1,\nu_2\}=1=\nu_3$ et $p\nmid
\Delta_{12}\Delta_{13}\Delta_{23}$, le point (3) et  la premi\`ere 
in\'egalit\'e du point (5)   du Lemme
\ref{rho*premier} fournissent
\begin{align*}\rho(p^{\nu_1},p^{\nu_2},p^{\nu_3})&=\rho^*(p^{\nu_1},p^{\nu_2},p^{\nu_3})
+\rho^*(1,1,1) p^{2(\nu_1+\nu_2) }\\&=\rho^*(1,1,1) p^{2(\nu_1+\nu_2) }
\ll  p^{2(\nu_1+\nu_2) }.\end{align*}
Le cas $p=2$ se traite de la m\^eme mani\`ere.

D\'emontrons le point (4).
Lorsque $\lceil \dm\nu_3\rceil \leq \min\{\nu_1,\nu_2\}$, puisque
$$\rho(p^{\nu_1},p^{\nu_2},p^{\nu_3})\leq p^{2 \nu_3 
}\rho(p^{\nu_1},p^{\nu_2},1)$$
il suffit de supposer $\nu_3=0$. Nous nous pla\c cons de plus dans le 
cas $\nu_2\leq \nu_1$.
Nous partons de la relation issue du Lemme \ref{rho*premier}
$$\rho(p^{\nu_1},p^{\nu_2},1)
= \sum_{0\leq k\leq  \nu_1 }
\rho^*\big(p^{ \nu_1-k },p^{\max\{\nu_2-k,0\}},1\big)
p^{2  \min\{\nu_2,k\}  }.$$
La contribution des termes associ\'es \`a $k\geq \nu_2$ est clairement
$$\leq p^{\nu_1+2\nu_2}\sum_{k\geq \nu_2}p^{-k}(1-1/p)= p^{\nu_1+ \nu_2}.$$
Lorsque $p\nmid \Delta_{12}$, les autres termes sont nuls. Lorsque 
$p\mid \Delta_{12}$
la contribution  
des termes associ\'es \`a $k< \nu_2$ est
\begin{align*}
  \sum_{\max\{0,\nu_2-v_p(\Delta_{12})\}\leq k<  \nu_2 }
\rho^*\big(p^{ \nu_1-k },p^{ \nu_2-k  },1\big)
p^{2  k  }
\ll
p^{\nu_1+  \nu_2+\min\{ \nu_2,v_p(\Delta_{12})\}}
,\end{align*}
ce qui prouve le point (4).

Pour \'etablir le point (5), nous nous pla\c cons dans le cas   $ 
\nu_2\leq \lceil\dm\nu_3\rceil\leq
\nu_1$ et remarquons
$$\rho(p^{\nu_1},p^{\nu_2},p^{\nu_3})\leq p^{2\nu_2} 
\rho(p^{\nu_1},1,p^{\nu_3}).$$
Dans la somme $$\rho(p^{\nu_1},1,p^{\nu_3})
= \sum_{0\leq k\leq   \nu_1 }
\rho^*\big(p^{ \nu_1-k },1,p^{\max\{\nu_3-2k,0\}} \big)
p^{2  \min\{\nu_3,2k\}  },$$
la contribution des $k\geq \lceil\dm\nu_3\rceil $ est
$$\leq  \sum_{k\geq \lceil \nu_3/2\rceil }p^{\nu_1+ 2\nu_3-k }(1-1/p)=
p^{  \nu_1+\nu_3+[ \nu_3/2] }.$$
Les autres termes de la somme sont nuls si $p\nmid \Delta_{13}$ ou si $p\mid
\Delta_{13}$ et $\min\{\nu_1-k,\nu_3-2k\}>
v_p(\Delta_{13})$. Nous nous pla\c cons donc dans le cas $p\mid
\Delta_{13}$ et sommons sur les $k$ tels que $\min\{\nu_1-k,\nu_3-2k\}\leq
v_p(\Delta_{13})$

Lorsque
$p\mid
\Delta_{13}$ et
$\nu_1\geq \nu_3$ (ie $\nu_3-2k\leq \nu_1-k$ pour tout $k\geq 0$), la 
contribution des $k<
\lceil\dm\nu_3\rceil
$ est
\begin{align*}  \sum_{0\leq k<   \lceil \nu_3/2\rceil }
\rho^*\big(p^{ \nu_1-k },1,p^{  \nu_3-2k } \big)
p^{4k  }&\ll \sum_{\max\{ 0, (\nu_3-v_p(\Delta_{13}))/2\}  
\leq k< 
\lceil\nu_3/2\rceil }  
p^{\nu_1+2\nu_3- k } \\&
\ll   p^{ \nu_1  +\nu_3+[\nu_3/2]+\min\{ \lceil \nu_3/2\rceil, \lceil 
v_p(\Delta_{13})/2\rceil\}}
.\end{align*}

Lorsque $p\mid \Delta_{13}$ et
$\nu_1< \nu_3$ (ie $\nu_3-2k\geq \nu_1-k$ seulement pour tout $0\leq 
k\leq \nu_3-\nu_1$), la
contribution des
$k<
\lceil\dm\nu_3\rceil
$ est
\begin{align*}  &\sum_{0\leq k<   \lceil \nu_3/2\rceil }
\rho^*\big(p^{ \nu_1-k },1,p^{  \nu_3-2k } \big)
p^{4k  }\\&\ll \sum_{\max\{ \nu_3-\nu_1, (\nu_3-v_p(\Delta_{13}))/2\}\leq  
k<   \lceil \nu_3/2\rceil } p^{
  \nu_1+2\nu_3-k}\\&\quad +\sum_{\max\{ 0,  \nu_1-v_p(\Delta_{13}) 
\}\leq k<  \nu_3-\nu_1 } p^{
  2\nu_1+\nu_3+\min\{[v_p({\rm disc}( Q))/2],[\nu_3/2]\}} \\&
\ll
p^{ \nu_1  +\nu_3+[\nu_3/2]+\min\{ \lceil \nu_3/2\rceil, \lceil 
v_p(\Delta_{13})/2\rceil\}}
\\&\quad+
\nu_3 p^{ \nu_1+\nu_3+[\nu_3/2]+\min\{\nu_1-[\nu_3/2],\lceil 
v_p(\Delta_{13})/2\rceil \}+
\min\{ [\nu_3/2],[v_p({\rm disc}( Q))/2]\}}
,\end{align*} puisque $\nu_1\leq [\nu_3/2]+\lceil 
v_p(\Delta_{13})/2\rceil$ est une condition
n\'ecessaire pour que la deuxi\`eme  somme ne soit pas nulle. Cela 
cl\^ot la d\'emonstration.
\end{proof}

Ce qu'il faut retenir de ce lemme est que la fonction 
$$ 
f(\ma{d}):=
\frac{\rho(\ma{d})}{d_1d_2d_3}
$$ 
 ressemble \`a la fonction $r$ 
: $\ma{d}\mapsto r(\ma{d})=r_{{\rm
disc}( Q)}(d_3)$, o\`u
$$r_{{\rm disc}( Q)}(d ):=\sum_{k\mid d}\chi_{{\rm disc}( Q)}(k) $$
et 
$\chi_{{\rm disc}( Q)}(n):=(\frac{{\rm disc}( 
Q)}{n})$ 
est le symbole de Kronecker.
Nous notons $h$ la fonction  satisfaisant les relations
$$f(\ma{d}) 
=(h*r)(\ma{d})=\sum_{\colt{\ma{k}\in\N^3}{k_j\mid 
d_j}}
h\Big(\frac{d_1}{k_1},\frac{d_2}{k_2},\frac{d_3}{k_3}\Big)r(\ma{k}).$$

Le r\'esultat dont nous aurons besoin est le suivant.

\begin{lem}\label{sumhlog}Soit $\varepsilon>0$. Lorsque   $L_1, L_2,Q 
$ sont des formes  satisfaisant (H2),
nous avons
$$\sum_{\ma{k}\in\N^3}\frac{|h(\ma{k})|\log (k_1k_2k_3)}{ 
k_1k_2k_3}\ll L_\infty^\varepsilon.$$
En particulier lorsque $\sigma_p$ est d\'efini par \eqref{defsigma}, on a
$$\prod_{p}\sigma_p =L(1,\chi_{{\rm disc}( 
Q)})\sum_{\ma{k}\in\N^3}\frac{ h(\ma{k}) }{ k_1k_2k_3}
\ll L_\infty^\varepsilon.$$
\end{lem}

\begin{proof}
La fonction $h$ est multiplicative puisque $\rho$ et $r$ le sont. 
Notant $P=P(L_1,L_2,Q)$ la somme \`a
majorer, nous avons donc $P\leq P'/\log 2$ avec
\begin{equation}\label{majP}
P'=P'(L_1,L_2,Q):=
\prod_{p }\Big(1+\sum_{\mnu\in\Z_{\geq 0}^3}\frac{|h
(p^{\nu_1},p^{\nu_2},p^{\nu_3})|\log(p^{\nu_1+ \nu_2+ \nu_3})
}{p^{ \nu_1+ \nu_2+ \nu_3}}\Big).\end{equation}

Nous commen\c cons par montrer que nous pouvons nous restreindre au 
cas de formes primitives.   Soit
$P'_p=P'_p(L_1,L_2,Q)$ le facteur relatif \`a un nombre premier du 
produit \eqref{majP}.
Lorsque  $p\nmid  \ell_1\ell_2q$, le Lemme 1 permet d'\'ecrire
$$\rho(p^{n_1},p^{n_2},p^{n_3};L_1,L_2,Q)=
\rho(p^{n_1},p^{n_2},p^{n_3};L_1^*,L_2^*,Q^*) 
$$
et ainsi
\begin{equation}\label{e0P'nonprim}\begin{split}P'(L_1,L_2,Q)&= \prod_{p\nmid
\ell_1\ell_2q}
P_p'(L_1^*,L_2^*,Q^*)\prod_{p\mid
\ell_1\ell_2q}P'_p(L_1,L_2,Q)\\
&=  
P'(L_1^*,L_2^*,Q^*)\prod_{p\mid
\ell_1\ell_2q}\frac{P'_p(L_1,L_2,Q)}{P'_p(L_1^*,L_2^*,Q^*)}.
\end{split}
\end{equation}
L'inverse $r^{(-1)}$ de la fonction $r$ au
  sens de la convolution est donn\'ee par
  $$r^{(-1)}(d_1,d_2,d_3)=\mu(d_1)\mu(d_2)   \vartheta(d_3) .$$
avec $\vartheta$ la fonction multiplicative d\'efini par
\begin{equation*} 
\vartheta( p^n ):=\left\{\begin{array}{ll}
\mu(p^n)(1+\chi_{{\rm disc}( Q)}(p^n)), &\mbox{si  $n=1$ ou  $n\neq 3$},\\
\chi_{{\rm disc}( Q)}(p ) ,  &\mbox{si $n=2$}.
\end{array}
\right.
\end{equation*}
  L'in\'egalit\'e $|r^{(-1)}(d_1,d_2,d_3)|\leq r(d_1,d_2,d_3)$ fournit
\begin{equation}\label{inegh}|h
(p^{\nu_1},p^{\nu_2},p^{\nu_3})|=|(f*r^{(-1)})
(p^{\nu_1},p^{\nu_2},p^{\nu_3})|\leq (f*r)
(p^{\nu_1},p^{\nu_2},p^{\nu_3}).\end{equation}
De \eqref{inegh}, nous obtenons l'in\'egalit\'e
\begin{equation}\label{majP'P''}P'_p
  \leq c_p\Big(1+\sum_{\mnu\in\Z_{\geq 0}^3}\frac{ 
r(p^{\nu_1},p^{\nu_2},p^{\nu_3})  \log(p^{\nu_1+
\nu_2+ \nu_3}) }{p^{ \nu_1+ \nu_2+ \nu_3}}\Big)P''_p
\end{equation}
avec $c_2=1/\log 2$ et $c_p=1$ si $p>2$ et
$$P''_p=P''_p(L_1,L_2,Q):= 1+\sum_{\mnu\in\Z_{\geq 0}^3}\frac{\rho 
(p^{\nu_1},p^{\nu_2},p^{\nu_3})
\log(p^{\nu_1+ \nu_2+ \nu_3}) }{p^{2( \nu_1+ \nu_2+ \nu_3)}} .$$
Il est clair que
$$\prod_{p\mid \ell_1\ell_2q}\Big(1+\sum_{\mnu\in\Z_{\geq 0}^3}\frac{
r(p^{\nu_1},p^{\nu_2},p^{\nu_3})\log(p^{\nu_1+
\nu_2+ \nu_3}) }{p^{ \nu_1+ \nu_2+ \nu_3}}\Big)\ll L_\infty^\varepsilon.$$
D'apr\`es le Lemme \ref{nonprimitive}, on a
$$\frac{\rho
(p^{\nu_1+v_p(\ell_1)},p^{\nu_2+v_p(\ell_2)},p^{\nu_3+v_p(q)})
}{p^{2( \nu_1+ \nu_2+ \nu_3+v_p(\ell_1)+v_p(\ell_2) +v_p(q))}}=
\frac{\rho
(p^{\nu_1 },p^{\nu_2 },p^{\nu_3 };L_1^*,L_2^*,Q^*)
}{p^{2( \nu_1+ \nu_2+ \nu_3 )}}.$$
Il vient
$$
\frac{\rho
(p^{\nu_1 },p^{\nu_2 },p^{\nu_3 };L_1 ,L_2 ,Q )
}{p^{2( \nu_1+ \nu_2+ \nu_3 )}}\leq \frac{\rho
(p^{ \nu_1'},p^{ \nu_2'},p^{ \nu_3'};L_1^*,L_2^*,Q^*)
}{p^{2(\nu_1'+ \nu_2'+ \nu_3')}}.
$$
avec 
$$
\nu_i':=\max\{\nu_i-v_p(\ell_i),0\},\quad
\nu_3':=\max\{\nu_3-v_p(q),0\},
$$ 
pour $i=1,2$. 
Lorsque les $\nu_j'$ sont fix\'es, les nombres de triplets $\mnu$ 
satisfaisant 
cela 
est au plus  
$(v_p(\ell_1)+1)(v_p(\ell_2)+1)(v_p(q)+1)$,  
et pour ces triplets compt\'es nous avons lorsque $\nu_1'+ \nu_2'+ 
\nu_3'\geq 1$
$$\nu_1+ \nu_2+ \nu_3\leq (\nu_1'+ \nu_2'+ 
\nu_3')(v_p(\ell_1)+1)(v_p(\ell_2)+1)(v_p(q)+1).$$
Nous obtenons
\begin{align*}P''_p &\leq c_p 
((v_p(\ell_1)+1)(v_p(\ell_2)+1)(v_p(q)+1))^2\log p
\\&\quad\times\Big(1+\sum_{\mnu'\in\Z_{\geq 0}^3}\frac{\rho
(p^{ \nu_1'},p^{ \nu_2'},p^{ \nu_3'};L_1^*,L_2^*,Q^*) \log 
(p^{\nu_1'+ \nu_2'+ \nu_3'})
}{p^{2(\nu_1'+ \nu_2'+ \nu_3')}}\Big)
\\&\leq c_p^2 ((v_p(\ell_1)+1)(v_p(\ell_2)+1)(v_p(q)+1))^2(\log p)  
P'_p(L_1^*,L_2^*,Q^*)
\\&\quad\times
\Big(1+\sum_{\mnu\in\Z_{\geq 0}^3}\frac{ 
r(p^{\nu_1},p^{\nu_2},p^{\nu_3})  \log(p^{\nu_1+
\nu_2+ \nu_3}) }{p^{ \nu_1+ \nu_2+ \nu_3}}\Big),
\end{align*}
ce qui permet  d'obtenir via \eqref{e0P'nonprim} et \eqref{majP'P''}
$P'(L_1,L_2,Q)\ll  L_\infty^\varepsilon P'(L_1^*,L_2^*,Q^*).$

Nous supposons maintenant que les formes $L_1,L_2,Q$ sont primitives 
et majorons~$P'$.
Nous avons lorsque   $\nu\in\N$ les
relations
$ h(p^\nu,1,1)=h(1,p^\nu,1)=0 $ issues du Lemme \ref{rhopremier}(1), 
et lorsque de plus $p\nmid
2 {\rm disc}( Q)$
$$
h(1,1,p)=-\frac{1}{p}\chi_{{\rm disc}( Q)}(p ),\qquad
h(1,1,p^\nu)\ll \nu+1,$$
issu du Lemme \ref{rhopremier}(2) et du calcul de $r^{(-1)}(1,1,p^k).$
Ainsi la contribution dans la somme $P'_p$ des exposants $\mnu$ tels que
$\#\{i\,:\, \nu_i\geq1\}\leq 1$  est lorsque
$p\nmid 2 {\rm disc}( Q) $ \'egale \`a $  O\big( {\log p}/{p^2}
\big).$  Du Lemme \ref{rhopremier} et de \eqref{inegh}, nous 
d\'eduisons  la majoration
\begin{align*}
|h (p^{\nu_1},p^{\nu_2},p^{\nu_3})|&\leq 2
\sum_{\colt{\ma{n}\in\Z_{\geq 0}^3}{n_j\leq
\nu_j}}f(p^{n_1},p^{n_2},p^{n_3})
\\&\leq 2
  \prod_{i=1}^3(\nu_i+1)\max_{\colt{\ma{n}\in\Z_{\geq 0}^3}{n_j\leq
\nu_j}}f(p^{n_1},p^{n_2},p^{n_3})
\\&\ll \prod_{i=1}^3(\nu_i+1)^2
   p^{\nu_1+\nu_2+\nu_3+k_p-\max\{\nu_1+\nu_2,\nu_1+\lceil 
\nu_3/2\rceil,\nu_2+\lceil
\nu_3/2\rceil,\nu_3\}},
\end{align*}
o\`u
$k_p=[v_p({\rm disc}( Q))/2]+v_p(\Delta_{12}\Delta_{13}\Delta_{23})$.
Nous avons utilis\'e pour la premi\`ere majoration l'in\'egalit\'e
$|r^{-1}(p^{k_1},p^{k_2},p^{k_3})|\leq 2 .$ Lorsque   $ 
\max\{\nu_1,\nu_2\}= 1=\nu_3$ et $p\nmid 2 {\rm
disc}( Q)\Delta_{12}\Delta_{13}\Delta_{23}$, on utilisera
$$
|h (p^{\nu_1},p^{\nu_2},p^{\nu_3})|\ll  p^{\nu_1+\nu_2+\nu_3 -2}.
$$

Cela permet de majorer la contribution des $\mnu$ tels que 
$\#\{i\,:\, \nu_i\geq1\}\geq 2$.
Nous avons ainsi
$$\prod_{p\nmid 2 {\rm disc}( Q)\Delta_{12}\Delta_{13}\Delta_{23} }P'_p\ll 1.$$
Il est clair que
$$\prod_{p\mid 2 {\rm disc}( Q)\Delta_{12}\Delta_{13}\Delta_{23} 
}\Big(1+\sum_{\mnu\in\Z_{\geq 0}^3}\frac{
r(p^{\nu_1},p^{\nu_2},p^{\nu_3})\log(p^{\nu_1+
\nu_2+ \nu_3}) }{p^{ \nu_1+ \nu_2+ \nu_3}}\Big)\ll L_\infty^\varepsilon.$$
Compte tenu de \eqref{majP'P''}, il reste \`a montrer  
\begin{equation}\label{majprod}\prod_{p\mid 2 {\rm disc}( 
Q)\Delta_{12}\Delta_{13}\Delta_{23} }P''_p\ll
L_\infty^\varepsilon.\end{equation}

La contribution \`a la somme $P''_p$ aux indices satisfaisant
$\max\{\nu_1,\nu_2\}\leq \lceil\dm\nu_3\rceil$
est d'apr\`es le point (3) du Lemme \ref{rhopremier}
$$ 
\ll \log p\sum_{\nu_3\geq 1} (\nu_3+1)^2p^{\min\{[v_p({\rm 
disc}(
Q))/2],[\nu_3/2]\}-\nu_3}\ll  \frac{\log p}{p}.$$
La contribution \`a la somme $P''_p$ aux indices satisfaisant
$\max\{\lceil \dm\nu_3\rceil,1\} \leq  \nu_2\leq \nu_1$ est d'apr\`es 
le point (4) du Lemme
\ref{rhopremier}
\begin{align*}&   \ll \log p\sum_{  1\leq \nu_2\leq \nu_1} (\nu_1+1)
  p^{  \min\{ \nu_2, v_p(\Delta_{12})\}-\nu_1-\nu_2}\\&\ll \log 
p\sum_{  \nu_2\geq 1} (\nu_2+1)
  p^{  \min\{ \nu_2, v_p(\Delta_{12})\}-2\nu_2}\ll \frac{\log p}{p}.
\end{align*}
De m\^eme lorsqu'on intervertit les indices $1$ et $2$. Enfin les cas 
$\max\{ \nu_j,1\}\leq
\lceil\dm\nu_3\rceil\leq  \nu_i$ se majore de la m\^eme mani\`ere \`a 
partir du point (5) du
Lemme~\ref{rhopremier}. Enfin le cas o\`u deux des $\nu_j$ sont nuls 
s'appuie sur les points (1) et (2)
du Lemme~\ref{rhopremier}. Ainsi, nous obtenons
$P''_p=1+O(\log p/p).$  
Cela permet de montrer \eqref{majprod}. 
\end{proof}

\section{Interpr\'etation de la constante} 
\lab{sectionconstante}
  \subsection{Densit\'e $p$-adique}
Lorsque $A\in\Z$ o\`u $\alpha=v_p(A)$, $\lambda\in \Z_{\geq 0}$ et 
$n\in\N$, nous consid\'erons
\begin{equation*}  
   S_\lambda(A;p^n):=\#\{(x, y) \in(\Z/p^n\Z)^2: p^{\lambda}xy\equiv A 
\smod{p^n}\}.
\end{equation*}
Si $\alpha\leq n$, il est facile de montrer
\begin{equation*} 
S_\lambda(A;p^n)=\left\{\begin{array}{ll}
p^{2n}, &\mbox{si $\lambda\geq n$},\\
  p^{n+\lambda}(1-1/p)(1+\alpha-\lambda),
& \mbox{si $\lambda<n$}.
\end{array}
\right.
\end{equation*}

Posant
\begin{align*} 
M_{ \mnu}(p^n):=\#\{\x \in(\Z/p^n\Z)^2\,: \, v_p(L_i(\x))=\nu_i,\, 
v_p(Q(\x))=\nu_3\},\\
M_{ \mnu}'(p^n):=\#\{\x \in(\Z/p^n\Z)^2\,:\, 
v_p(L_i(\x))\geq\nu_i,\, v_p(Q(\x))\geq\nu_3\},
\end{align*}
on a lorsque $n\geq \nu_1+\nu_2+\nu_3$
\begin{align*}
M_{ \mnu}'(p^n)&=p^{2n-2\nu_1-2\nu_2-2\nu_3}\rho( 
p^{\nu_1},p^{\nu_2},p^{\nu_3})\\
M_{
\mnu}(p^n)&=\sum_{\ma{e}\in \{0,1\}^3}(-1)^{e_1+e_2+e_3}M_{
\mnu+\ma{e}}'(p^n).\end{align*}
Posant $m_j=\max\{\lambda_j,\mu_j\}$, nous obtenons
\begin{align*}N_{\mal,\mmu}(p^n)&=
p^{3n+\la_1+\la_2+\la_3}(1-1/p)^3\sum_{m_j\leq 
\nu_j<n}M_{\mnu}(p^n)\prod_{1\leq
   j\leq 3}(1+\nu_j-\lambda_j)\\
&\quad+O(n^3p^{4n}).
\end{align*}
Faisant le changement de variables $n_j=\nu_j+e_j-\lambda_j$ et notant que
$\nu_j+e_j\geq m_j+e_j\geq m_j$, nous en d\'eduisons
\begin{align*}
\sigma_{\mal,\mmu}(p)
=&\Big(1-\frac{1}{p}\Big)^3\sum_{\colt{\ma{n}\in\Z_{\geq 
0}^3}{n_j\geq m_j-\lambda_j}}
\frac{\rho (p^{\la_1+n_1},p^{\la_2+n_2}, 
p^{\la_3+n_3})}{p^{2(\la_1+n_1+\la_2+n_2+\la_3+n_3)}} \\
&\quad \times
\sum_{0\leq e_j\leq \min\{1,\la_j+n_j-m_j\}} (-1)^{e_1+e_2+e_3}
\prod_{1\leq j\leq 3}(1+n_j-e_j).
\end{align*}
Or nous avons
$$
\sum_{0\leq e\leq \min\{1,\la+n-m\}}
\hspace{-0.2cm}
(-1)^{e}
(1+n-e)=
\left\{
\begin{array}{ll}
1,& \mbox{si $\la+n-m\geq 1$},\\
1+m-\la,& \mbox{si $\la+n-m=0$},
\end{array}
\right.
$$ ce qui fournit donc
\begin{align*}
\sigma_{\mal,\mmu}(p)
&=\Big(1-\frac{1}{p}\Big)^3\sum_{\ma{n}\in\Z_{\geq 0}^3}
\frac{\rho (p^{\max\{m_1,\la_1+n_1\}}, \ldots,
p^{\max\{m_3,\la_3+n_3\}})}{p^{2(\max\{m_1,\la_1+n_1\} + \ldots+ 
\max\{m_3,\la_3+n_3\})}} \\
&=\Big(1-\frac{1}{p}\Big)^3\sum_{\ma{n}\in\Z_{\geq 0}^3}
\frac{\rho (p^{\max\{\mu_1,\la_1+n_1\}}, \ldots,
p^{\max\{\mu_3,\la_3+n_3\}})}{p^{2(\max\{\mu_1,\la_1+n_1\} + \ldots+
  \max\{\mu_3,\la_3+n_3\})}} ,
\end{align*}
et qui permet ainsi d'obtenir la densit\'e $p$-adique introduite en 
\eqref{defsigmad}.

  \subsection{Densit\'e archim\'edienne}
La vari\'et\'e d\'efinie en \eqref{defvar} est le lieu des z\'eros 
des polyn\^omes
$f_i(\x,\ma{s},\ma{t})$ d\'efinis par
$$f_i(\x,\ma{s},\ma{t})=L_i(\x)/d_i-s_it_i,\quad 
f_3(\x,\ma{s},\ma{t})=Q(\x)/d_3-s_3t_3.$$
On calcule le d\'eterminant suivant
$$
\det\left(
\begin{array}{ccc}
\frac{\partial f_1}{\partial s_1}&\frac{\partial f_2}{\partial 
s_1}&\frac{\partial f_3}{\partial s_1}\\
& &\\
\frac{\partial f_1}{\partial s_2}&\frac{\partial f_2}{\partial 
s_2}&\frac{\partial f_3}{\partial s_2}\\
& &\\
\frac{\partial f_1}{\partial s_3}&\frac{\partial f_2}{\partial 
s_3}&\frac{\partial f_3}{\partial s_3}
\end{array}
\right)
=\det\left(
\begin{array}{ccc}
-t_1&0&0\\
& &\\
0&-t_2&0\\
& &\\
0&0&-t_3
\end{array}
\right)
=-t_1t_2t_3.
$$
Les variables $t_i$ et $s_i$ appartiennent \`a $[1,+\infty)$.
Maintenant, nous pouvons d\'efinir
\begin{equation}
\omega_\infty(\mcal{R},V):=\lim_{X\to+\infty}   
\frac{1}{X^2(\log r'X)^3}\int\!\!\!\int_{\x\in X\mcal{R}}
\int\!\!\!\int\!\!\!\int_{\mathbf{t}\in T} \frac{\d t_1\d t_2\d t_3}{t_1t_2t_3}\d x_1\d x_2,
\label{defomega}
\end{equation}
o\`u  
$$ 
T:=  \left\{\ma{t}\,:~
1\leq t_i \leq L_i(\x),~ 1\leq t_3\leq Q(\x),~
\Big(\frac{\log t_1}{\log
r'X},\frac{\log t_2}{\log r'X},\frac{\log t_3}{2\log r'X}\Big)\in V 
\right\}.
$$
Lorsque $X$ tend vers l'infini, on a
\begin{align*} 
V(X):=&\int \!\!\!\int_{\x\in X\mcal{R}}
\int\!\!\!\int\!\!\!\int_{\mathbf{t}\in T}\frac{\d t_1\d t_2\d 
t_3}{t_1t_2t_3}\d x_1\d x_2
\\  = &\int \!\!\!\int_{\x\in X\mcal{R}'(X)}
\int\!\!\!\int\!\!\!\int_{\ma{u}\in V   '} 2(\log r'X)^3 {\d u_1\d 
u_2\d u_3}\, \d x_1\d
x_2
\end{align*}
o\`u $\mcal{R}'(X)=\{\x\in \mcal{R}\,:\, L_i(\x)\geq 1/X,\, Q(\x)\geq 1/X^2\}$
et $V   ':=\{ \ma{u}\in V   \,:\, u_i\leq \log L_i(\x)/(\log 
r'X),\,u_3\leq \log Q(\x)/(2\log r'X)\}$.
Un changement de variable homoth\'etique fournit alors
\begin{align*}V(X)
& = X^2
\int\!\!\!\int_{\x\in  \mcal{R}'(X)}\log  (XL_1(\x)) \log 
(XL_2(\x)) \log (X^2 Q(\x)) \d x_1\d x_2\\
& = 2X^2  (\log X)^3\vol(\mcal{R}'(X))\vol(V   ) +O\big(X^2  (\log 
X)^2 I(X)\big)
\end{align*}
avec $$I(X)=
\int\!\!\!\int_{\x\in  \mcal{R}'(X)}(1+|\log L_1(\x)|+|\log 
L_2(\x)|+| \log Q(\x)|) \d x_1\d x_2.$$
Il est clair que 
$\vol(\mcal{R}'(X))=\vol(\mcal{R})+o(1)  $ et l'int\'egrale~$I(X)$ est born\'ee. Cela fournit bien
$\omega_\infty(\mcal{R},V)=2 \vol(\mcal{R})\vol(V   ).$

\section{D\'emonstration du   Th\'eor\`eme \ref{main0}}\lab{sectionmain0}

Avant de commencer la d\'emonstration, nous notons que le Corollaire 
1 de~\cite{nair} fournit la
majoration
\begin{equation}
    \label{majSXtriv}S(X)\ll L_\infty^\varepsilon r_\infty^2X^2(\log 
X)^3.\end{equation}
Nous serons amen\'es \`a utiliser les majorations
\begin{equation}
    \label{prelim}
r'/(2L_\infty)\leq r_\infty\leq 2r'L_\infty,\qquad \vol(\mcal{R})\leq 
4r_\infty^2 .\end{equation}
La majoration $r'\leq 2r_\infty L_\infty$ est \'evidente ; la 
majoration de $r_\infty$ provient des
relations
$$x_1=\frac{b_2L_1(\x)-b_1L_2(\x)}{b_2a_1-b_1a_2 } ,
\quad x_2=\frac{a_2L_1(\x)-a_1L_2(\x)}{ b_1a_2-b_2a_1},
$$
  o\`u nous avons utilis\'e la notation \eqref{defforme}.
Notre d\'emonstration reprend les id\'ees d\'evelopp\'ees dans \cite{h-b03}
et \cite{4linear}. La premi\`ere \'etape est un  cas particulier d'un 
r\'esultat d\^u \`a
Marasingha
\cite{M08} concernant le niveau de r\'epartition que nous \'enon\c 
cons de la mani\`ere
suivante.

\begin{lem}\label{LOD}
Soit $\varepsilon>0$. Lorsque  $L_1, L_2,Q,\mcal{R}$ satisfont 
(H1)--(H3), $X\geq1$, $V_1,V_2,V_3 \ge
2$ et
$V =V_1 V_2 V_3,$
il existe une constante absolue $A>0$ telle que
\begin{align*} 
  \sum_{\colt{\ma{d}\in\N^3}{d_i\le V_i}}
\left|\#\big(\mathsf{\Lambda}({\ma{d}})\cap X\mcal{R}_{\ma{d}}\big)
  - \vol(\mcal{R}_{\ma{d}})X^2 \frac{\rho(\bd) }{(d_1d_2d_3)^2} \right|
   \ll  L_\infty^{\varepsilon}(r_\infty X\sqrt{V}+V)(\log V)^{A} ,
\end{align*}
o\`u $\mcal{R}_{\ma{d}}\subseteq \mcal{R}$ d\'esigne une r\'egion 
convexe d\'ependante de $\ma{d}$.
\end{lem}

\begin{proof}Dans \cite{M08}, la d\'ependance par rapport aux formes 
$L_1$, $L_2$ et $Q$ n'est pas
indiqu\'ee, mais il est ais\'e de l'obtenir. De m\^eme, nous avons 
ajout\'e la possibilit\'e d'avoir une
r\'egion de comptage pouvant d\'ependre de ${\ma{d}}$ ce qui ne 
modifie pas la preuve. Nous ne donnons pas
plus de d\'etails.

Nous nous contentons d'indiquer comment on peut \'etendre ce 
r\'esultat \`a des formes
non primitives.
Fixant $d'$ et $\ell$, le nombre de $d$ tel que $d'=d/(d,\ell)$ est 
inf\'erieur \`a $\tau(\ell)$. Nous
avons  
$\mathsf{\Lambda}({\ma{d }};L_1 ,L_2 ,Q 
)=\mathsf{\Lambda}({\ma{d'}};L_1^*,L_2^*,Q^*).$
  Ainsi,
d'apr\`es le Lemme \ref{nonprimitive}, le terme \`a majorer est
born\'e par  
\begin{align*} 
\tau(\ell_1)\tau(\ell_2)\tau(q)
  \sum_{\colt{\ma{d'}\in\N^3}{d_i'\le V_i}}
\left|\#\big(\mathsf{\Lambda}({\ma{d'}};L_1^*,L_2^*,Q^*)\cap 
X\mcal{R}_{\ma{d'}}\big)
  - \vol(\mcal{R}_{\ma{d'}})X^2 \frac{\rho(\bd';L_1^*,L_2^*,Q^*)) 
}{(d_1'd_2'd_3')^2} \right|
\end{align*}
La majoration du Lemme \ref{LOD}  dans le cas de formes primitives 
permet de conclure.
\end{proof}
\goodbreak

   Posons $X'=r'X$ et
$Y=(r'X)^{1/2}/(\log X)^C$ o\`u $C$ est une constante que nous 
pr\'eciserons \`a la fin.
Avec les notations \eqref{defforme},
un des deux coefficients $a_1$ ou $a_2$ est non nul puisque les $L_i$ 
ne sont pas proportionnelles.
Supposons
$a_1\neq 0$, le cas
$a_1=0$ et
$a_2\neq 0$ se traitant de la m\^eme mani\`ere.

Nous majorons les sommes $S_0(X)$ et $S'_0(X)$ d\'efinies par
\begin{align*} 
  S_0(X)&:=\sum_{\x\in\Z^2\cap X\mcal{R}} 
\tau(L_2(\x))\tau(Q(\x))\sum_{\colt{d_1\mid L_1(\x)}{ Y
<d_1<2L_\infty{X'}/Y}}1,
\\
S_0'(X)&:=\sum_{\colt{\x\in\Z^2\cap X\mcal{R}}{L_1(\x)\leq Y^2}} 
\tau(L_2(\x))\tau(Q(\x))\tau(L_1(\x)).
\end{align*}

\begin{lem}\label{lemmajS0X}Avec les hypoth\`eses du Th\'eor\`eme \ref{main0},
on a
\begin{align*}S_0(X)\ll   L_\infty^{ \varepsilon}r_\infty r'    X^2 
(\log X)^{2}
\log\log X,\quad
S_0'(X)\ll   L_\infty^{ \varepsilon}r_\infty r'    X^2 (\log X)^{3-2C}.
\end{align*}
\end{lem}

\begin{proof}
Soit $\tau'$ la fonction multiplicative d\'efinie par
\begin{equation*} 
\tau'(p^\nu)=\left\{\begin{array}{ll}
2, &\mbox{si $\nu=1$},\\
  (\nu+1)^2,
& \mbox{si $\nu\geq 2$}.
\end{array}
\right.
\end{equation*}
Ainsi nous avons pour tout $(n_1,n_2)\in\N^2$ les in\'egalit\'es
$\tau(n_1)\tau(n_2)\leq \tau'(n_1n_2)$ et $\tau'(n_1)\leq
\tau(n_1)^2.$
Avec cette notation, nous \'ecrivons
$$ 
S_0(X)\leq \sum_{d_1\in [Y',2L_\infty {X'}/Y)}\sum_{ \substack{\x\in \Z^2\cap X\mcal{R}\\d_1\mid
L_1(\x)}}
\tau'(L_2(\x) Q(\x)).$$
Rappelant l'hypoth\`ese $a_1\neq 0$, nous faisons le changement de 
variables $d_1u_1=L_1(\x)$ et
$u_2=x_2$.  Nous avons 
$a_1^3(L_2  Q)(\x)=(L_2  Q)(d_1u_1-b_1u_2,u_2)=F_{d_1}(u_1,u_2)$,
o\`u
$$F_{d_1}(u_1,u_2)=(a_2d_1u_1+\Delta_{12}u_2)(a_3d_1^2u_1^2+d_1(a_1c_3-2a_3b_1)u_1u_2+\Delta_{13}u_2^2),
$$
avec $\Delta_{12}$ et $\Delta_{13}$ d\'efinis en \eqref{defdisc}. 
Ainsi,
$$
S_0(X)\leq \sum_{ d_1\in [Y',2L_\infty {X'}/Y) }\sum_{ 
\tolt{(u_1,u_2)\in\Z^2}{u_1\leq
r'X/d_1}{u_2\leq r_\infty X}}
\tau'(F_{d_1}(u_1,u_2)).
$$
La  forme $F_{d_1}(u_1,u_2)$ a des
coefficients dont le pgcd divise
$\Delta_{12}\Delta_{13}$. Le Lemme 1 de~\cite{nair} permet d'affirmer 
que le discri\-minant de
$F_{d_1}(u_1,u_2)$ est \'egal \`a $d_1^{6}c$ o\`u $c$ est un entier 
qui s'\'ecrit comme un polyn\^ome des
coefficients des formes $L_1$, $L_2$ et $Q$.
Le Th\'eor\`eme 1 de \cite{nair}
fournit alors
$$S_0(X)\ll L_\infty^\varepsilon \sum_{ d_1\in [Y',2L_\infty {X'}/Y) 
}\phi^*(d_1)\frac{r'r_\infty
X^2(\log r'X)^2}{d_1}$$
o\`u $\phi^*(m)=\prod_{p\mid m}(1+1/p)$. Cela permet d'obtenir la 
majoration de $S_0(X)$  annonc\'ee
puisque
\begin{equation*} 
\log r'X\ll \log  X
\end{equation*}
sous la condition $ r'X^{1-\varepsilon}\geq 1$.

Il est clair que la majoration
$$S_0'(X)\leq \sum_{ d_1\leq X' }\sum_{ 
\tolt{(u_1,u_2)\in\Z^2}{u_1\leq Y^2/d_1}{u_2\leq r_\infty X}}
\tau'(F_{d_1}(u_1,u_2)).$$
permet de la m\^eme mani\`ere d'obtenir la majoration de $S_0'(X)$  annonc\'ee.
\end{proof}

Nous imposons $2C\geq 1$. 
Dans la suite, nous noterons implicitement
$$
\mdelta:=\Big(\frac{\log d_1}{\log
X'},\frac{\log d_2}{\log X'},\frac{\log d_3}{2\log X'}\Big).
$$
Le Lemme \ref{lemmajS0X} permet de se restreindre au cas $L_i(\x)> 
Y^2$, puis $d_1\leq Y$ et $d_2\leq
{X'}^{1/2}$. En effet, si $d_1\geq \sqrt{L_1(\x)}$, alors 
$L_1(\x)/d_1$ est un diviseur de $L_1(\x)$
inf\'erieur \`a $\sqrt{L_1(\x)}.$ Ensuite, la majoration de la somme 
$S_0(X)$ permet de majorer la
contribution des~$\ma{d}$ tels que $Y<d_1\leq \sqrt{L_1(\x)}\leq 
2L_\infty X'/Y.$ De plus puisque
$ |\log L_i(\x)-\log X'|\leq  \log(2L_\infty)+\log (X'/Y^2)$ changer 
$d_1$ en $L_1(\x)/d_1$ revient \`a
changer $\delta_1=\log d_1/\log X'$ en $1-\delta_1$.
Nous faisons le m\^eme raisonnement pour imposer $d_2\leq
{X'}^{1/2}$. Cela justifie ainsi la restriction $V   \subseteq 
[0,\dm]^2\times[0,1]$.

Nous obtenons
$$S(X;V)= S_1(X)+S_2(X) +O( L_\infty^{ \varepsilon}r_\infty r'    X^2 
(\log X)^{2}\log\log X),$$
avec
\begin{align*} 
S_1(X)&:= \sum_{{  d_1\leq Y}}\sum_{{d_2\leq {X'}^{1/2}}}
\sum_{\colt{d_3\leq {X'} }{\mdelta\in V 
}}\#\big(\mathsf{\Lambda}({\ma{d}})\cap X\mcal{R}\big),
\\
S_2(X)&:=\sum_{\x \in \Z^2\cap
X\mcal{R}}\sum_{\colt{d_1\mid L_1(\x)}{  d_1\leq 
Y}}\sum_{\colt{d_2\mid L_2(\x)}{d_2\leq {X'}^{1/2}}}
\sum_{\tolt{d_3\mid Q(\x)}{d_3< |Q(\x)|/{X'} 
}{(\delta_1,\delta_2,Q(\x)/(2\log X')-\delta_3)\in V   }}
\hspace{-0.3cm}
1.
\end{align*}

Nous estimons $S_1(X)$ gr\^ace au Lemme \ref{LOD}.
Il vient
\begin{align*}S_1(X)=& \vol(\mcal{R})X^2\sum_{{  d_1\leq 
Y}}\sum_{{d_2\leq {X'}^{1/2}}}
\sum_{\colt{d_3\leq {X'} }{\mdelta\in V   }} \frac{\rho(\bd) 
}{(d_1d_2d_3)^2}\\&
+O\Big(
L_\infty^{ \varepsilon}    X^2 \Big(\frac{r_\infty r' }{(\log 
X)^{C/2-A}}+\frac{{r'}^2 }{(\log
X)^{C -A}}
\Big)\Big)
\end{align*}
La conjonction~\eqref{majSXtriv} et
\eqref{prelim} permet de montrer que le terme d'erreur de~\eqref{estS(X)} est
$O(L_\infty^\varepsilon r_\infty^2X^2(\log X)^3).$  
En prenant  $C\geq 2$, nous pouvons donc  supposer que $ r'\ll 
r_\infty (\log X)^{C/2}$. En choisissant
$C\geq 2A$, nous obtenons
$$
S_1(X)= \vol(\mcal{R})X^2M(X;V)+O( L_\infty^{ \varepsilon}r_\infty 
r'   X^2 (\log
X)^{2}\log\log X),
$$
avec
\begin{equation}\label{defMX} 
M(X;V):=\sum_{{  d_1\leq Y}}\sum_{{d_2\leq {X'}^{1/2}}}
\sum_{\colt{d_3\leq {X'} }{\mdelta\in V   }} \frac{\rho(\bd) 
}{(d_1d_2d_3)^2} .
\end{equation}

Nous estimons $S_2(X)$ gr\^ace au Lemme \ref{LOD} appliqu\'e \`a la r\'egion
$$ 
\mcal{R}(r'd_3/X):=\{ \x\in \mcal{R}\,:\quad
|Q(\x)|>r'd_3/X \}.
$$ 
Nous pourrons nous restreindre aussi \`a 
l'ensemble $\x\in X\mcal{R}({r'}^2/(\log
X)^2)$ de sorte que nous puissions remplacer la condition
$ (\delta_1,\delta_2,Q(\x)/(2\log X')-\delta_3)\in V   $ par 
$(\delta_1,\delta_2,1-\delta_3)\in V$.

Montrons la majoration
\begin{equation}\label{majvol}
  \vol(\mcal{R}\smallsetminus\mcal{R}(\alpha))\ll   r_\infty
\sqrt{\alpha},\end{equation} avec  $\alpha:=r'd_3/X$.
Nous avons $a_3b_3\neq 0$ avec la notation \eqref{defforme}. Nous 
pouvons supposer sans perte de
g\'en\'eralit\'e
$a_3>0$ et donc
$a_3\geq 1$. Nous faisons le changement de variable
$x_1'=x_1+c_3x_2/2a_3$, $x_2'=x_2$. Lorsque
  $\x\in \mcal{R}\smallsetminus\mcal{R}(\alpha)$, les nouvelles variables
v\'erifient
$$-\alpha\leq a_3{x_1'}^2-b_3'{x_2'}^2\leq \alpha$$
avec $b_3':=-(4b_3a_3-c_3^2)/4a_3$.
Lorsque $b_3'{x_2'}^2\leq  \alpha $, on a $|x_1'|\leq \sqrt{2\alpha}$ 
et $|x_2'|\leq r_\infty$ ce qui
fournit une contribution \`a \eqref{majvol}  inf\'erieure \`a $2 
r_\infty\sqrt{2\alpha}$. Lorsque
$b_3'{x_2'}^2> \alpha$, on a
$\sqrt{a_3}|x_1'|\in [v^-,v^+]$ avec
$$v^-:=\sqrt{ b_3'{x_2'}^2-\alpha },\quad v^+:=\sqrt{ b_3'{x_2'}^2+\alpha }$$
On a
$$ v^+-v^-=\frac{{v^+}^2-{v^-}^2}{v^+ +v^- }\leq  \sqrt{2\alpha},$$ ce qui
fournit une contribution \`a \eqref{majvol}  encore  inf\'erieure \`a 
$ 2 r_\infty\sqrt{2\alpha}$.
Cela cl\^ot la preuve de \eqref{majvol}.

En se restreignant
encore
\`a
$r'\ll r_\infty (\log X)^{C/2}$, nous obtenons
\begin{align*}
S_2(X)=& X^2\sum_{{  d_1\leq Y}}\sum_{{d_2\leq {X'}^{1/2}}}
\sum_{\colt{d_3\leq {X'} }{(\delta_1,\delta_2,1-\delta_3)\in V   }}
\frac{\rho(\bd)\vol(\mcal{R}(r'd_3/X))}{(d_1d_2d_3)^2}\\& +O\Big(
L_\infty^{ \varepsilon}    X^2 \Big(\frac{r_\infty r' }{(\log 
X)^{C/2-A}}+\frac{{r'}^2 }{(\log
X)^{C -A}} +r_\infty r' (\log
X)^{2}
\Big)\Big)
\\=& \vol(\mcal{R})X^2M(X;V')+O\big(r_\infty r'   X^2(R(X) + 
L_\infty^{ \varepsilon}  (\log
X)^{2}\log\log X)\big),
\end{align*}
avec
$$R(X)=  \sum_{{  d_1\leq {X'}^{1/2}}}\sum_{{d_2\leq {X'}^{1/2}}}
\sum_{{d_3\leq {X'} }} \frac{\rho(\bd) \sqrt{d_3/X'} }{(d_1d_2d_3)^2} $$
et  $V'$ l'image de $V$ par la transformation 
$\delta_3\leftrightarrow 1-\delta_3$. 
Nous montrons
$R(X)\ll    L_\infty^{ \varepsilon}     (\log
X)^{2} .$  
En effet, nous \'ecrivons $\rho(\ma{d})/d_1d_2d_3=(h*r)(\ma{d})$ et 
nous intervertissons les sommes. Il
vient
\begin{align*}R(X)&\leq  \sum_{\tolt{\ma{k}\in\N^3}{  k_i\leq 
{X'}^{1/2}}{k_3\leq {X'} }}
\frac{|h(\ma{k})|}{k_1k_2k_3}
  \sum_{\tolt{\ma{m}\in\N^3}{   m_i\leq {X'}^{1/2}/k_i}{m_3\leq {X'}/k_3 }}
\frac{r(\ma{m}) \sqrt{m_3k_3/X'} }{m_1m_2m_3} ,
\end{align*}
avec $i=1,2$.
La somme int\'erieure en $\ma{m}$ est major\'ee par 
$O(L_\infty^\varepsilon(\log X)^2)$ ce qui fournit la
majoration souhait\'ee pour $R(X)$ gr\^ace au Lemme \ref{sumhlog}.
  En rassemblant ces r\'esultats, nous
obtenons
\begin{align*}S (X;V)= 
\vol(\mcal{R})X^2\big(M(X;V)+M(X;V')\big)+O( L_\infty^{ 
\varepsilon}
r_\infty  r'    X^2 (\log X)^{2}\log\log X).\end{align*}

Pour estimer $M(X;V)$ d\'efini en \eqref{defMX}, nous utilisons un 
lemme classique concernant la moyenne
de convol\'ee de fonctions arithm\'etiques.

\begin{lem}\label{convol} Soient $g$ et $h$ deux fonctions 
arithm\'etiques et $C,C',C''$ trois
cons\-tantes telles que
$$\sum_{d=1}^\infty \frac{|h(d)|\log (2d)}{ d}\leq C'',\quad 
\sum_{d\leq x} \frac{g(d)}{ d}=C\log
x+O(C').$$ Alors on a uniform\'ement  l'estimation
$$\sum_{n\leq x}\frac{(g*h)(n)}{n}=C\log x\sum_{d=1}^\infty \frac{ 
h(d)}{ d}+O(C''(C+C')).$$
Soit $v$ une fonction de $[0,1]$ dans $\R$ continument d\'erivable 
par morceaux de d\'eriv\'ee born\'ee.
Alors on a uniform\'ement  l'estimation
$$\sum_{n\leq x}\frac{(g*h)(n)}{ n}v\Big(\frac{\log n}{\log 
x}\Big)=C\log x\sum_{d=1}^\infty \frac{ h(d)}{
d}\int_0^1v(t)\d t+O(C''(C+C')).$$
\end{lem}

Nous ne r\'edigeons pas tous les d\'etails de la preuve de 
l'estimation de $M(X;V)$.  
Nous renvoyons le
lecteur \`a la section \ref{sectionrho}  
  et reprenons  certaines des notations de cette partie.
Lorsque $\varepsilon>0$, on a
$$\sum_{n\leq x} \frac{r_{ {\rm disc}(Q)} (n)}{ n}=L(1,\chi_{{\rm 
disc}( Q)} )\log
x+O(L_\infty^{\varepsilon})  
$$
avec $\chi_{{\rm disc}( Q)}=(\frac{{\rm disc}( 
Q)}{n})$. 
Nous  appliquons trois fois le Lem\-me~\ref{convol} pour estimer   
la sommation en $d_1$, $d_2$ et
$d_3$ en choisissant pour
$g$ la fonction constante \'egale \`a $1$ pour les deux premi\`eres 
et  la fonction $r_{{\rm disc}(Q)} $
pour la troisi\`eme fois. Nous notons $1_V   $ la fonction 
caract\'eristique des $\mdelta\in V$.

Nous avons
\begin{align*}M(X;V)&=
\sum_{{  d_1\leq Y}}\sum_{{d_2\leq {X'}^{1/2}}}
\sum_{{d_3\leq {X'} }} \frac{(h*r)(\ma{d})}{ d_1d_2d_3 }1_V 
\Big(\frac{\log d_1}{\log X'},\frac{\log
d_2}{\log X'},\frac{\log d_3}{2\log X'}\Big) \\
&=
\sum_{{d_2\leq {X'}^{1/2}}}
\sum_{{d_3\leq {X'} }}\sum_{\colt{e_2\mid d_2}{e_3\mid 
d_3}}\frac{r_{{\rm disc}(Q)}(e_3)}{d_2d_3}
\\&\quad \times  
\sum_{d_1\leq Y}\sum_{e_1\mid 
d_1}\frac{h(d_1/e_1,d_2/e_2,d_3/e_3)}{d_1}
1_V    \Big(\frac{\log d_1}{\log X'},\frac{\log
d_2}{\log X'},\frac{\log d_3}{2\log X'}\Big) .
\end{align*}
Or la somme en $d_1$ et $e_1$ vaut, d'apr\`es  le Lemme \ref{sumhlog} 
et le Lemme
\ref{convol},
\begin{align*}v\Big(\frac{\log
Y}{\log X'}&,\frac{\log
d_2}{\log X'},\frac{\log d_3}{2\log X'}\Big) \log X'
\sum_{k_1\in\N}\frac{h(k_1,d_2/e_2,d_3/e_3)}{k_1}\\&\quad+O\Big(
\sum_{k_1\in\N}\frac{|h(k_1,d_2/e_2,d_3/e_3)|\log (2k_1)}{k_1}
\Big),\end{align*}
o\`u  
$v(x,t_2,t_3):=\int_{t_1\leq x}1_V     (t_1,t_2,t_3 ) \d t_1,$
donc
\begin{align*}  
M(X;V)=&\log X'
\sum_{k_1\in\N}\sum_{{d_2\leq {X'}^{1/2}}}
\sum_{{d_3\leq {X'} }}v\Big(\frac{1}{  2},\frac{\log
d_2}{\log X'},\frac{\log d_3}{2\log X'}\Big)
\\& \times \sum_{\colt{e_2\mid 
d_2}{e_3\mid d_3}}
\frac{r_{{\rm disc}(Q)}(e_3)}{d_2d_3}\frac{h(k_1,d_2/e_2,d_3/e_3)}{k_1}
+O(L_\infty^{ \varepsilon}
(\log X')^2\log\log X),
\end{align*}
o\`u nous avons utilis\'e l'estimation
$$ 0\leq v(\dm,t_2,t_3)-v(x,t_2,t_3) \leq \dm-x.$$
En r\'eit\'erant deux fois cette manipulation, nous obtenons
\begin{align*}M(X;V)
&= {2(\log {X'})^3}\vol(V \cap   [0,\dm]^3 )\prod_p \sigma_p+ O( 
L_\infty^{ \varepsilon}
(\log X)^{2}\log\log X) \end{align*}
et
\begin{align*}
M(X;V')
&= {2(\log {X'})^3}\vol(V' \cap   [0,\dm]^3 )\prod_p \sigma_p+ O( 
L_\infty^{ \varepsilon}
(\log X)^{2}\log\log X). 
\end{align*}
En utilisant la majoration \eqref{prelim} pour  $\vol(\mcal{R})$, et 
en observant que
$$\vol(V \cap   [0,\dm]^3 )+\vol(V' \cap   [0,\dm]^3 )=
\vol(V \cap   [0,\dm]^2\times[0,1] )=\vol(V),$$ 
nous obtenons l'estimation
recherch\'ee au Th\'eor\`eme~\ref{main0}.

\section{D\'emonstration du   Th\'eor\`eme \ref{maind}}\lab{sectionmaind}
Notre d\'emarche est, bien entendu, d'appliquer le Th\'eor\`eme \ref{main0}.
La premi\`ere \'etape permet de se ramener au cas o\`u 
$(D_i,\ell_i)=1$ et $(D_3,q)=1$.
Avec la notation \eqref{ecritprim}, nous avons
$ 
\mathsf{\Lambda}(\ma{D};L_1,L_2,Q)=\mathsf{\Lambda}(\ma{D'};L_1^*,L_2^*,Q^*),$
avec la notation \eqref{defD'}.
Posant
\begin{equation}\label{defla*}\mathsf{\Lambda}^*({\ma{D'}}):=\{
\x\in\mathsf{\Lambda}({\ma{D'}};L_1^*,L_2^*,Q^*)\,:\, 
(x_1,x_2,D_1'D_2'D_3')=1\},
\end{equation}
nous avons
$$
\mathsf{\Lambda}({\ma{D'}})= \bigsqcup_{b\mid
\psi(\ma{D'})}b\mathsf{\Lambda}^*({\ma{D''}};L_1',L_2',Q')
$$ 
et
$$
  \psi(\ma{D'})=\prod_{p\mid 
D_1'D_2'D_3'}p^{\max\{v_p(D_1'),v_p(D_2'),\lceil v_p(D_3')/2 \rceil \}}
$$
o\`u
\begin{equation*} 
\ma{D''}=\Big(\frac{D_1'}{(D_1',b)},
\frac{D_2'}{(D_2',b)},
\frac{D_3'}{(D_3',b^2)}\Big)
\end{equation*}
et
$$\quad L_i'=\frac{bL_i}{(d_i,b)},\quad
Q'=\frac{b^2Q}{(d_3,b^2)}.$$
Il en d\'ecoule
$$S(X,\ma{d},\ma{D};V)=\sum_{b\mid \psi(\ma{D'})}\sum_{\x\in 
\mathsf{\Lambda}^*({\ma{D''}} )\cap
(X/b)\mcal{R}}
\tau\Big(\frac{L_1'(\x)}{d_1'},\frac{L_2'(\x)}{d_2'},\frac{Q'(\x) 
}{d_3'}; V\Big)$$
o\`u
$$\ma{d'}=\Big(\frac{d_1}{(d_1,b)},
\frac{d_2}{(d_2,b)},\frac{d_3}{(d_3,b^2)}\Big).$$

Soient $D'=D_1'D_2'D_3'$ et $D''=D_1''D_2''D_3''$.  
Nous partitionnons en classes 
d'\'equivalence disjointes.
Lorsque $\x$ et $\y\in \Z^2$ satisfont $(\x,D'')=(\y,D'')=1$, nous 
d\'efinissons la relation
d'\'equivalence par $\x\sim \y$ si et seulement s'il existe 
$\lambda\in\Z$ tel que
$\x \equiv \lambda  \by \pmod{D''}.$
C'est une relation d'\'equivalence et les $\lambda$ doivent 
v\'erifier $(\lambda,D'')=1$.
Nous notons $\cU(D'') $ cet ensemble de classes d'\'equivalence et 
pour tout   $\by \in
\cA$ fix\'e   avec $\cA \in \cU(D'')$, il est facile de v\'erifier
$$
\cA = \{\x \in \Z^2 : \x \equiv a \by \smod{D''}\text{ avec }a\in\Z
\text{ o\`u }(a,D'') = 1 \}.
$$
Posant
$$\cV(\ma{D''})=\{\cA\in \cU(D'')\,:\, \cA\subseteq 
\mathsf{\Lambda}^*({\ma{D''}})\},$$
nous d\'efinissons pour $\cA\in \cV(\ma{D''})$ et $\y_0\in \cA$
$$
   G(\cA) = \{ \x \in \Z^2 : \exists a \in \Z
   \text{ tel que }  \x \equiv
   a \by_0 \smod{D''}  \}.$$
L'ensemble $
   G(\cA)$ est un r\'eseau de dimension $2$ de d\'eterminant $D''$ et
  $\cA = \{ \x \in G(\cA) : (\x , D'') =
1 \}$. Nous pouvons donc \'ecrire
\begin{equation}\label{e1}
S(X,\ma{d},\ma{D};V)=\sum_{b\mid \psi(\ma{D'})}\sum_{\cA\in \cV(\ma{D''})}
\sum_{e\mid D''}\mu(e)T(X,\cA,e;V)
\end{equation}
o\`u
$$ 
T(X,\cA,e;V):=\sum_{\x\in G_e(\cA) \cap
(X/b)\mcal{R} }
\tau\Big(\frac{L_1'(\x)}{d_1'},\frac{L_2'(\x)}{d_2'},\frac{Q'(\x) 
}{d_3'};V\Big)$$
avec
\begin{align*}G_e(\cA)&=G (\cA)\cap \{ \x\in \Z^2\,:\, e\mid \x\}
\\&= \{ \x\in \Z^2\,:\,
\exists a \in e\Z  \text{ tel que }  \x \equiv  a \by_0 \smod{D''}  \}.
\end{align*}
Nous sommes ainsi ramen\'es \`a un comptage sur un r\'eseau de 
d\'eterminant $D''e$.
Pr\'eparons maintenant l'utilisation du Th\'eor\`eme \ref{main0} pour 
estimer $T(X,\cA,e;V)$.  
Pour cela, nous
d\'efinissons $E$ un changement de va\-riables de sorte que
$$\x\in G_e(\cA)\Longleftrightarrow \x=E\v \text{ avec } \v\in\Z^2,   $$
o\`u $E=(\ma{e}_1,\ma{e}_2)$  
est la matrice de changement dans une base 
minimale de vecteurs. Ainsi $|\ma{e}_1|$ et
$|\ma{e}_2|$ sont les minima successifs des normes des \'el\'ements 
du r\'eseau $G_e(\cA)$. Puisque
$G_e(\cA)\subseteq \delta(\ma{D}),$ nous avons
\begin{equation}\label{reseau}
\delta(\ma{D})\leq |\ma{e}_1|\leq |\ma{e}_2|,\qquad 
|\ma{e}_1|.|\ma{e}_2|\gg \det G_e(\cA)=D''e.\end{equation}

  Posant $\mcal{R}'_E:=\{
\v\in\R^2\,:\, E\v\in
\mcal{R} /b\}$, nous avons
$$\vol( \mcal{R}'_E)=\frac{\vol( \mcal{R})}{b^2\det E}= \frac{\vol( 
\mcal{R})}{b^2D''e}.$$
Nous pouvons \'ecrire
$$T(X,\cA,e;V)=\sum_{\v\in \Z^2\cap
  \mcal{R}'_E }
\tau(M_1(\v),M_2(\v),M_3(\v) ;V)$$
avec
$M_i(\v)={L_i'(E\v)}/{d_i'}$,   $M_3(\v)={Q'(E\v)}/{d_3'}$.
Les quantit\'es intervenant dans le terme d'erreur du Th\'eor\`eme 
\ref{main0} satisfont
les in\'egalit\'es suivantes. D'une part
\begin{align*}r'(\ma{M},\mcal{R}'_E)&=
\sup_{\v\in\mcal{R}'_E}\max\{|M_1(\v)|,|M_2(\v)|,\sqrt{|M_3(\v)|}\}\\
&=\sup_{\x\in\mcal{R}/b 
}\max\{|L_1'(\ma{x})|/d_1',|L_2'(\ma{x})|/d_2',\sqrt{|Q'(\ma{x})|/d_3'}\}
\\
&= 
\sup_{\x\in\mcal{R}}\max\{|L_1(\ma{x})|/d_1,|L_2(\ma{x})|/d_1,\sqrt{|Q(\ma{x})|/d_3}\}\\&
= r'_{\ma{d}}(L_1,L_2,Q,\mcal{R})=  r'_{\ma{d}},
\end{align*}
o\`u
\begin{equation}\label{defr'd}r'_{\ma{d}}
:=\sup_{\x\in\mcal{R}}\max\{|L_1(\ma{x})|/d_1,|L_2(\ma{x})|/d_2,\sqrt{|Q(\ma{x})|/d_3}\} 
.
\end{equation}
D'autre part
\begin{align*}L_\infty(\ma{M} )=
\max\{\|M_i\|\}&\leq \max\{\|L_1\|,\|L_2\|,\|Q\|\}\|E\|
\\&\ll L_\infty(L_1,L_2, Q)D''e\leq  L_\infty {D}^2.
\end{align*}
Enfin
\begin{align*}r_\infty( \mcal{R}'_E)&=
\sup_{\v\in\mcal{R}'_E}\max\{|v_1|,|v_2|\}
\ll \frac{r_\infty( \mcal{R} )}{b}=\frac{r_\infty }{b}.
\end{align*}
Posant
$$ 
\sigma_p(E):=\Big(1-\frac{1}{p}\Big)^3
\sum_{\mnu\in\Z_{\geq 0}^3} \frac{\rho
(p^{\nu_1},p^{\nu_2},p^{\nu_3},\ma{M})}{p^{2\nu_1+2\nu_2+2\nu_3}},$$
le Th\'eor\`eme \ref{main0} fournit lorsque 
$r'_{\ma{d}}X^{1-2\varepsilon }\geq 1$
\begin{align*}
  T(X,\cA,e;V)=&2 \vol( \mcal{R})\vol(V)W(E)X^2(\log (r'_{\ma{d}}X))^3
\\&+
O\Big(\frac{L_\infty^{ \varepsilon}D^{ \varepsilon} }{b } (r_\infty 
r'_{\ma{d}}+r_\infty^2)    X^2
(\log X)^{2}
\log\log X\Big),\end{align*}
avec
$$W(E):=\frac{1}{b^2D''e} \prod_{p}\sigma_p(E).$$

Notre t\^ache maintenant est de montrer que le nombre de termes somm\'es dans le
membre de droite de  \eqref{e1} est petit.  
\begin{lem} On a
$$
\#\cV(\ma{D})\leq 8^{\omega(D_1D_2D_3)}
a(\ma{D},\mDelta),
$$
avec la notation \eqref{def1aDD}.
\end{lem}

\begin{proof} 
La relation
$$\#\cV(\ma{D})=\frac{\rho^*(\ma{D})}{\varphi(D_1D_2D_3)}=
\prod_{p^{\nu_i}\parallel
D_i}\frac{\rho^*(p^{\nu_1},p^{\nu_2},p^{\nu_3})}{\varphi(p^{\nu_1+\nu_2+\nu_3})}$$
permet de se restreindre \`a des puissances d'un nombre premier $p$.
De plus, d'apr\`es le Lemme \ref{rho*premier}, nous avons $$
v_p(\Delta_{12})\geq \min\{\nu_1,\nu_2\}, \qquad
v_p(\Delta_{i3})\geq \min\{\nu_i,\nu_3\}$$ \`a moins que
$\rho^*(p^{\nu_1},p^{\nu_2},p^{\nu_3})=0$. 
Examinons deux cas, les autres s'obtenant de mani\`ere analogue.
Lorsque $\nu_1\geq \nu_2\geq \nu_3$, le Lemme \ref{rho*premier} fournit 
$$\frac{\rho^*(p^{\nu_1},p^{\nu_2},p^{\nu_3})}{\varphi(p^{\nu_1+\nu_2+\nu_3})}\leq
p^{ \nu_3+\min\{\nu_1,\nu_2, v_p(\Delta_{12})\}}=p^{ \nu_3+ \nu_2} ,$$ ce qui
convient.
 Lorsque $   \nu_3>\max\{\nu_1,\nu_2\}$, le Lemme \ref{rho*premier} fournit 
$$\frac{\rho^*(p^{\nu_1},p^{\nu_2},p^{\nu_3})}{\varphi(p^{\nu_1+\nu_2+\nu_3})}\leq
8p^{ \nu_1+ \nu_2} p^{\min\{[v_p({\rm disc}( Q))/2],[\nu_3/2]\}} 
,$$ ce qui convient encore.
\end{proof}

En reportant dans \eqref{e1}, nous obtenons sous la condition 
$r'_{\ma{d}}X^{1-2\varepsilon }\geq 1$
\begin{equation}
  \begin{split}
S(X,\ma{d},\ma{D};V)=& 2 \vol( \mcal{R})\vol( V)W X^2(\log 
(r'_{\ma{d}}X))^3\\&+ O\big(
L_\infty^{\varepsilon} D^{ 2\varepsilon }
a'(\ma{D},\mDelta) 
(r_\infty 
r'_{\ma{d}}+r_\infty^2) X^2 (\log X)^{2}
\log\log X\big)\end{split}\label{SXdD}\end{equation}
avec
$$W:=\sum_{b\mid \psi(\ma{D'})}\sum_{\cA\in \cV(\ma{D''})}
\sum_{e\mid D''}\mu(e)W(E).$$

Nous introduisons la fonction multiplicative $\psi_0$   d\'efinie par
\begin{equation}\label{defpsi0} 
\psi_0(p^{\beta_1},p^{\beta_2},p^{\beta_3})=
  \max_{\beta \leq \max\{\beta_1,\beta_2,\lceil \beta_3/2 \rceil\}} p^{\min\{\beta,\beta_1\}+\min\{\beta,\beta_2\}+\min\{2\beta,\beta_3\}-2\beta}.
\end{equation}
On observe que lorsque $D_3$ est sans facteur carr\'e, on a
\begin{equation}\label{psi0sscarre}
\psi_0(D_1,D_2,D_3)= (D_1,D_2,D_3) .
\end{equation}
  De plus, nous avons l'in\'egalit\'e
\begin{equation}\label{psi01/2}
\psi_0(D_1,D_2,D_3)\leq  (D_1 D_2 D_3)^{1/2}.\end{equation}
cons\'equence de
$ \min\{\beta,\beta_1\}+\min\{\beta,\beta_2\}+\min\{2\beta,\beta_3\}\leq 2\beta
+(\beta_1+\beta_2+\beta_3)/2.$ Nous utiliserons
\begin{equation}\label{D''b2} D''b^2=
\frac{b^2D'}{(D_1',b)(D_2',b)(D_3',b^2)}
\geq   \frac{D'}{\psi_0(D_1',D_2',D_3')}.
\end{equation}
Observons maintenant que la majoration
\begin{equation}\label{majW}
W\ll L_\infty^{\varepsilon}D^{\varepsilon }\frac{\psi_0(D_1',D_2',D_3')}{D'}
a'(\ma{D},\mDelta)
\ll L_\infty^{\varepsilon}D^{\varepsilon }
a'(\ma{D},\mDelta)
\end{equation}
d\'ecoule de la deuxi\`eme majoration du Lemme \ref{sumhlog} et de la relation
\eqref{D''b2}.

La quantit\'e d\'efinie en \eqref{defr'd} v\'erifie
  $$
\frac{r'}{d_1d_2d_3}\leq
r'_{\ma{d}}
\leq r'.
$$
Puisque $W\ll L_\infty^{\varepsilon} D^{\varepsilon}a'(\ma{D},\mDelta)
$ d'apr\`es 
\eqref{majW}, on en d\'eduit qu'on peut
remplacer dans
\eqref{SXdD} la quantit\'e
$r'_{\ma{d}}$ par $r'$.  
Sous la condition suppl\'ementaire $r'_{\ma{d}}X^{1-2\varepsilon }\geq 1$
nous obtenons
\begin{equation}
  \begin{split}
S (X,\ma{d},\ma{D};V)
=& 2 \vol(\mcal{R})\vol(V)
  X^2(\log r'X)^3 W
  \\& +
O\Big(   {(DL_\infty)^{\varepsilon}}a'(\ma{D},\mDelta)
 (r_\infty r'+r_\infty^2)   
X^2 (\log X)^{2}
\log\log X\Big).
  \end{split}
  \label{est2S(X)d} \end{equation}

Nous constatons que, lorsque $d_1d_2d_3\leq X^{\varepsilon }$,  la 
condition $r'_{\ma{d}}X^{1-2\varepsilon }\geq 1$ est une
cons\'equence de
$r' X^{1-
\varepsilon }\geq 1$. Il reste \`a remarquer que si   $d_1d_2d_3> 
X^{\varepsilon }$ d'apr\`es \eqref{majSXtriv} nous avons
\begin{align*}
S (X,\ma{d},\ma{D};V) \ll    S(X)&\ll L_\infty^\varepsilon 
r_\infty^2X^2(\log X)^3\\
&\ll L_\infty^\varepsilon D^\varepsilon r_\infty^2X^2(\log 
X)^2(\log\log X).\end{align*}
ce qui fournit \eqref{est2S(X)d} dans ce cas-l\`a puisque $W\ll 
L_\infty^{\varepsilon} D^\varepsilon a'(\ma{D},\mDelta) $ 
d'apr\`es \eqref{majW}.

Pour obtenir le facteur  $1/\delta(\ma{D})$ dans le terme d'erreur, 
il suffit de constater que
$$S(X,\ma{d},\ma{D};L_1,L_2,Q,\mcal{R},V)=S(X,\ma{d},\ma{D};\delta 
L_1,\delta L_2,
\delta^2Q,\mcal{R}/\delta, V),$$
avec $\delta=\delta(\ma{D})$. La valeur de $r'$ est alors inchang\'ee et  la
quantit\'e~$r_\infty$ est alors divis\'ee par
$\delta(\ma{D})$. Enfin, il est clair que les constantes   
obtenues dans le terme principal en facteur
de $ X^2(\log r'X)^3$  
sont les m\^emes puisqu'elles ne d\'ependent pas de $X$.

Il est utile d'avoir une majoration uniforme de 
$  
S (X,\ma{d},\ma{D}):=
S (X,\ma{d},\ma{D};[0,1]^3).
$ 
Notre d\'emonstration fournit facilement la majoration suivante.

\begin{lem}\lab{majSXdD}
Soit $\varepsilon>0$. Lorsque  $L_1, L_2,Q,\mcal{R}$ satisfont 
(H1)--(H3), $\ma{d}$ et $\ma{D}$ tels que
$d_i\mid D_i$, on a
\begin{equation*}
  \begin{split}
S (X,\ma{d},\ma{D})\ll(DL_\infty)^{\varepsilon}
a'(\ma{D},\mDelta)
\Big( \frac{ \psi_0(D_1',D_2',D_3')}{D_1'D_2'D_3'}     (r_\infty X)^2 (\log
X)^{3}+ \frac{(r_\infty X)^{1+\varepsilon}}{\delta(\ma{D})}\Big),
  \end{split}  \end{equation*}
o\`u $\ma{D'}$ a \'et\'e d\'efini en \eqref{defD'} et $\psi_0$  a 
\'et\'e d\'efini en \eqref{defpsi0} et
satisfait \eqref{psi01/2}.
\end{lem}

\begin{proof} Dans toute la d\'emonstration, nous omettons de 
signaler la d\'e\-pendance en $V=[0,1]^3$. 
  Au vu de \eqref{e1}, il suffit de majorer $T(X,\cA,e;V)$.  
Posant
\begin{equation}\label{deftau''}
\tau''(p^\nu):=\left\{\begin{array}{ll}
2, & \mbox{si $\nu=1$},\\
(\nu+1)^3, & \mbox{si $\nu\geq 2$},
\end{array}\right.\end{equation}
nous avons
$$ 
T(X,\cA,e;V) 
\leq\sum_{\tolt{\v\in \Z^2 }{v_1\ll r_\infty 
X/(|\ma{e}_1|b)}{v_2\ll r_\infty
X/(|\ma{e}_2|b)}}
\tau''(M_1(\v) M_2(\v) M_3(\v) ).$$
Compte tenu de \eqref{reseau}, le Th\'eor\`eme 1 de \cite{nair} fournit alors
$$ 
T(X,\cA,e;V)
\ll \frac{(DL_\infty)^{\varepsilon}}{D''b^2} (r_\infty X)^2 (\log
X)^{3}+(DL_\infty)^{\varepsilon}\frac{(r_\infty X)^{1+\varepsilon}}{ 
\delta(\ma{D})}.$$
Ici, nous avons utilis\'e les relations
$$\delta(\ma{D},L_1,L_2,Q)=\delta(\ma{D'},L_1^*,L_2^*,Q^*)\leq 
b\delta(\ma{D''},L_1',L_2',Q').$$
Puisque le nombre de termes somm\'es dans \eqref{e1} est
$O(D^\varepsilon a'(\ma{D},\mDelta))$   
et compte tenu de
\eqref{D''b2}, nous obtenons bien la majoration annonc\'ee.
\end{proof}

La fin de la d\'emonstration du   Th\'eor\`eme \ref{maind} est 
consacr\'ee au calcul de $W$.
\'Etudions les termes somm\'es dans $\sigma_p(E)$.
Nous rappelons les notations \eqref{defNi} et introduisons les 
notations suivantes
\begin{align*}&  \mu'_i=v_p(D_i'),\quad\mu''_i=v_p(D_i''),  \quad
\lambda'_i=v_p(d_i'),\quad  \mu =v_p(D ),\quad \mu' =v_p(D' ),\\& 
\mu'' =v_p(D''),\quad
\varepsilon=v_p(e),\quad
\beta=v_p(b),\quad \nu=\nu_1+\nu_2+\nu_3\end{align*}
o\`u, par souci de clart\'e, nous avons omis d'indiquer la 
d\'ependance en $p$ de ces va\-luations.

  On a
$$\rho
(p^{\nu_1},p^{\nu_2},p^{\nu_3},\ma{M})=\#(\mathsf{\Lambda}'(p^{\nu_1},p^{\nu_2},p^{\nu_3})\cap
\mathcal{B}(p^{\nu_1+\nu_2+\nu_3}))$$
o\`u
$$\mathsf{\Lambda}'(p^{\nu_1},p^{\nu_2},p^{\nu_3})=  \{ \x \in \Z^2 : 
p^{\nu_i+\lambda_i} \mid p^\beta L_i(\x) ,
~ p^{\nu_3+\lambda_3}\mid p^{2\beta}Q(\x) \} $$
et
$$\mathcal{B}(p^{\nu}) =\{ E\v\, :\, 0\leq v_i<p^{\nu}\}.$$
Le th\'eor\`eme de la base adapt\'ee permet d'affirmer qu'il existe 
une base $(\ma{e}_1,\ma{e}_2)$ de~$\Z^2$ telle qu'il existe 
$(\delta_1p^{m_1},\delta_2p^{m_2})\in\N$  
pour lequel la 
famille
$(\delta_1p^{m_1}\ma{e}_1,\delta_2p^{m_2}\ma{e}_2)$ est une base de
$E\Z^2$ et $(\delta_1\delta_2,p)=1$. De plus, on a
$\delta_1\delta_2p^{m_1+m_2}=\det E=D''e$  
  et donc~$m_1+m_2=\mu'' +\varepsilon$. Nous pouvons remplacer 
$\mathcal{B}(p^{\nu_1+\nu_2+\nu_3})$ par
$$\{ \ma{w}=w_1\delta_1p^{m_1}\ma{e}_1+w_2\delta_2p^{m_2}\ma{e}_2\, 
:\, 0\leq w_j<p^{\nu}\}$$
puis par
\begin{align*}&\{ \ma{w}=w_1 p^{m_1}\ma{e}_1+w_2 p^{m_2}\ma{e}_2\, :\, 0\leq
w_i<p^{\nu}\}=\mathcal{B}'(p^{m_1+\nu},p^{m_2+\nu})
\end{align*}
o\`u
$$\mathcal{B}'(p^{k_1},p^{k_2})  :=\{ 
\ma{w'}=w_1'\ma{e}_1+w_2'\ma{e}_2\in E\Z^2\, :\, 0\leq w_j' 
<p^{k_j}\}.$$
Nous en d\'eduisons les relations
\begin{align*}\rho
(p^{\nu_1},p^{\nu_2},p^{\nu_3},\ma{M})&=\#\big(\mathsf{\Lambda}' 
{(p^{\nu_1},p^{\nu_2},p^{\nu_3})}\cap
\mathcal{B}'(p^{\nu+m_1},p^{\nu+m_2})\big)\\
&=\frac{\#\big(\mathsf{\Lambda}'  {(p^{\nu_1},p^{\nu_2},p^{\nu_3})}\cap
\mathcal{B}'(p^{\nu+m_1+m_2},p^{\nu+m_1+m_2})\big)}{p^{ m_1 +m_2}}
\\
&=\frac{\#\big(\mathsf{\Lambda}'  {(p^{\nu_1},p^{\nu_2},p^{\nu_3})}\cap
\mathcal{B}'(p^{\nu+\mu'' +\varepsilon },p^{\nu+\mu'' +\varepsilon 
})\big)}{p^{ \mu'' +\varepsilon }}.
\end{align*}
Nous r\'e\'ecrivons cette relation
\begin{align*}
\frac{\rho(p^{\nu_1},p^{\nu_2},p^{\nu_3},\ma{M})}{p^{2\nu_1+2\nu_2+2\nu_3+v_p(D''e)}} 
&=\frac{\#\big(\mathsf{\Lambda}' {(p^{\nu_1},p^{\nu_2},p^{\nu_3})}\cap
\mathcal{B}'(p^{\nu+\mu'' +\varepsilon},p^{\nu+\mu'' 
+\varepsilon})\big)}{p^{2(\nu+\mu''
+\varepsilon)}}
\\&= \frac{\#\big(\mathsf{\Lambda}' {(p^{\nu_1},p^{\nu_2},p^{\nu_3})}\cap
\mathcal{B}'(p^{\nu+\mu''+1 },p^{\nu+\mu''+1})\big)}{p^{2(\nu+\mu'' +1)}}.
\end{align*}
La derni\`ere \'egalit\'e est un passage subtil de la 
d\'emonstration. Pour $\varepsilon=1$, elle est claire
mais, pour
$\varepsilon=0$, il faut observer que  les conditions incluses dans 
l'ensemble 
$ \mathsf{\Lambda}' {(p^{\nu_1},p^{\nu_2},p^{\nu_3})}$  ne 
d\'ependent que de la valeur des coordonn\'ees modulo
$p^{\nu+\mu'' }$.
Nous sommes maintenant en mesure de sommer par rapport \`a $\cA\in 
\cV(\ma{D''})$  
et $e$. Comme les termes sont multiplicatifs, il suffit de sommer
$(-1)^{\varepsilon}\sigma_p(E)p^{-(v_p(D''e)+2\beta)}$ par rapport 
\`a $p^\beta$, $\cV(p^{\mu_1''},p^{\mu_2''},p^{\mu_3''})$ 
et
$p^{\varepsilon}$ avec $\beta\leq \max\{\mu_1' ,\mu_2' ,\lceil \mu_3' 
/2 \rceil\}=:B'$ et
$\varepsilon\leq
\max\{ 1,\mu''\}$.
Cette somme vaut $W=\prod_pw_p$ avec
\begin{align*}w_p:=
\Big(1-\frac{1}{p}\Big)^3\sum_{0\leq \beta\leq B'}
\sum_{\mnu\in\Z_{\geq 0}^3} \frac{\rho'
(p^{\nu_1},p^{\nu_2},p^{\nu_3};p^{\nu+\mu'' +1})}{p^{2(\nu+\mu'' +1+\beta) }}
\end{align*}
o\`u
$$ 
\rho'(p^{\nu_1},p^{\nu_2},p^{\nu_3};p^{k})=\#\{ \ma{w'}\in 
[0,p^{k})\,:\,
p^{N_i }\mid p^{\beta}L_i(\ma{w'}),~ p^{N_3}\mid p^{2\beta}Q(\ma{w'}),
~p\nmid \ma{w'}\}
$$
avec la notation \eqref{defNi}. La fonction
$\rho'(p^{\nu_1},p^{\nu_2},p^{\nu_3};p^{k})/p^{2k}$ ne d\'epend pas 
de $k$ pourvu que $k\geq \nu+1$.
Puisque $\mu''\leq \mu'+B'-\beta$,
  nous pouvons \'ecrire
\begin{align*}w_p:=
\Big(1-\frac{1}{p}\Big)^3
\sum_{\mnu\in\Z_{\geq 0}^3} \sum_{0\leq \beta\leq B'}\frac{\rho'
(p^{\nu_1},p^{\nu_2},p^{\nu_3};p^{\nu+\mu' 
+1+B'-\beta})}{p^{2(\nu+\mu'  +1+B') }}.
\end{align*}
En faisant un changement de variables $ p^{\beta}
\ma{w'}=\ma{w}$, la valeur de $\beta$ correspond \`a la valuation 
$p$-adique de 
$\big(p^\beta w_1'
,p^\beta w_2',\psi(p^{\mu_1'},p^{\mu_2'},p^{\mu_3'})\big).$
La somme int\'erieure en $\beta$ vaut donc
\begin{align*}&\frac{1}{p^{2(\nu+\mu' +1+B') }}
\#\{ \ma{w }\in [0,p^{\nu+\mu'  +1+B'})\,:\, p^{N_i }\mid
  L_i(\ma{w }),\, p^{N_3}\mid  Q(
\ma{w})\}
\\&\quad=
\frac{1}{p^{2(N_1+N_2+N_3) }}
\#\{ \ma{w }\in [0,p^{N_1+N_2+N_3})\,:\, p^{N_i }\mid
  L_i(\ma{w }),\, p^{N_3}\mid  Q(
\ma{w})\}\end{align*}
puisque les conditions de divisibilit\'e d\'efinissant cet ensemble 
ne d\'ependent que des valeurs des
coordonn\'ees de $\ma{w }$ modulo $p^{N_1+N_2+N_3}$. \'Etant donn\'e 
la notation \eqref{defNi}, nous
obtenons bien  $w_p=\sigma_p(\ma{d},\ma{D})$ la relation recherch\'ee.

\section{D\'emonstration du   Th\'eor\`eme \ref{tht}}\lab{sectiontht}

Dans toute cette partie, l'ensemble $V$ est toujours $[0,1]^3$ et 
nous n'indi\-quons plus ce choix.
Pour relier $T(X)$ au r\'esultat \'etabli au Th\'eor\`eme 
\ref{maind}, on commence par montrer une
formule d'\'eclatement qui pallie la non-compl\`ete 
multiplicativit\'e de la fonction $\tau$.
\begin{lem}\lab{lemsplit}
Lorsque $\ma{n}=(n_1,n_2,n_3)\in\N^3$, on a
$$\tau(n_1n_2n_3)=\sum_{\colt{\ma{d}\in\N^3}{d_id_j\mid
n_k}}\frac{\mu(d_1d_2)\mu(d_3)}{2^{\omega((d_1,n_1))+\omega((d_2,n_2))}}
\tau\Big(\frac{n_1}{d_2d_3}\Big)\tau\Big(\frac{n_2}{d_1d_3}\Big)\tau\Big(\frac{n_3}{d_1d_2}\Big),$$
o\`u $\{i,j,k\}=\{1,2,3\}$.
\end{lem}

\begin{proof}
Nous partons de la formule
\begin{equation}\label{splite1}
  \tau(nm)=\sum_{ d\mid (n,m) } \mu(d)
\tau\Big(\frac{n }{d}\Big)\tau\Big(\frac{m}{d}\Big).
\end{equation}
Celle-ci se montre en remarquant que les deux membres de 
l'\'egalit\'e sont des fonctions multiplicatives
qui sont identiques sur les couples de puissances de nombre premier.
Ainsi
\begin{equation*}
  \tau(n_1n_2n_3)=\sum_{ d\mid (n_1n_2,n_3) } \mu(d)
\tau\Big(\frac{n_1n_2 }{d}\Big)\tau\Big(\frac{n_3}{d}\Big).
\end{equation*}
Le nombre de mani\`ere d'\'ecrire $d=d_2d_1$ avec $d_2\mid n_1$ et 
$d_1\mid n_2$ lorsque $d$ est sans
facteur carr\'e est \'egal \`a 
$2^{\omega((d,n_1,n_2))}=2^{\omega((d_1,n_1))+\omega((d_2,n_2))}$.
Il vient
\begin{equation*}
  \tau(n_1n_2n_3)=\sum_{ d_1d_2 \mid n_3, d_1\mid n_1, d_2\mid n_2 
}\frac{\mu(d_1d_2)
}{2^{\omega((d_1,n_1))+\omega((d_2,n_2))}}
\tau\Big(\frac{n_1  }{ d_2 }\frac{ n_2 }{ 
d_1}\Big)\tau\Big(\frac{n_3}{ d_1d_2}\Big),
\end{equation*}
ce qui fournit la formule du lemme apr\`es application de 
\eqref{splite1}  avec $n=n_1/d_2$ et
$m=n_2/d_1$ pour calculer
$\tau( {n_1  }/{ d_2 }\times{ n_2 }/{ d_1}) $.
\end{proof}
\medskip

Le lemme suivant permet de minorer $\delta(\ma{D})$ la quantit\'e 
d\'efinie en \eqref{defdelta}.

\begin{lem}\lab{lemmindelta}
Lorsque $(D_1,D_3)$ et $(D_2,D_3)$ sont sans facteur carr\'e, le ppcm
$$ 
\Big[\frac{(D_1,D_3)}{(D_1,D_3 ,\Delta_{13} ) 
},\frac{(D_2,D_3)}{(D_2,D_3 ,\Delta_{23} ) },\frac{(D_1,D_2)}{(
D_1,D_2,\Delta_{12}) }
\Big]
$$ 
est un diviseur de $ \delta(\ma{D}).$
\end{lem}

\begin{proof} 
Nous rappelons la notation    
\eqref{defdisc}. Nous avons les implications
\begin{align*} 
&d\mid L_1(\x),\,d\mid L_2(\x)\Rightarrow d\mid \Delta_{12}(x_1,x_2)\Rightarrow
d/(d,\Delta_{12})\mid (x_1,x_2),
\\ &d\mid L_i(\x),\,d\mid Q(\x)\Rightarrow d\mid \Delta_{i3}(x_1,x_2)
\Rightarrow d/(d,\Delta_{i3})\mid
(x_1,x_2)
\end{align*}
si $|\mu(d)|=1$. Le r\'esultat d\'ecoule ais\'ement de ces
implications. 
\end{proof}

Appliquons maintenant le Lemme \ref{lemsplit} aux termes somm\'es dans $T(X)$.
Nous avons  
$$T(X)=\sum_{\ma{e}\in\N^3}
\mu(e_1e_2)\mu(e_3)
\!\!\!\!\!\!\!\!\!\!\!\!\!\sum_{\colt{\x\in  X\mcal{R}}{ 
\x\in\mathsf{\Lambda}(e_2e_3,e_1e_3,e_1e_2)
}}
\!\!\!\!\!\!\!\!\!\!\!\!\! \frac{1}{2^{\omega(k_1)+\omega(k_2)}}
\tau\Big(\frac{L_1(\x)}{e_2e_3}\Big)\tau\Big(\frac{L_2(\x)}{e_1e_3}\Big)\tau\Big(\frac{Q(\x)}{e_1e_2}\Big)$$
o\`u $k_1=(e_1,L_1(\x))$ et $k_2=(e_2,L_2(\x))$.
Il vient alors
\begin{align*}
T(X)=\sum_{\ma{e}\in\N^3} 
\mu(e_1e_2)\mu(e_3)\sum_{\colt{\ma{k}=(k_1,k_2,k_1',k_2')\in\N^4}{k_ik_i'\mid
e_i}}
\frac{\mu(k_1')\mu(k_2')}{2^{\omega(k_1)+\omega(k_2)}} T_{\ma{e},\ma{k} }(X),
\end{align*}
avec
\begin{equation*}
  T_{\ma{e},\ma{k} }(X)=\sum_{ \x\in\mathsf{\Lambda}({\ma{e},\ma{k} })\cap
X\mcal{R}}
\tau\Big(\frac{L_1(\x)}{e_2e_3}\Big)
\tau\Big(\frac{L_2(\x)}{e_1e_3}\Big)\tau\Big(\frac{Q(\x)}{e_1e_2}\Big),
\end{equation*}
o\`u 
$\mathsf{\Lambda}({\ma{e},\ma{k}}):=\mathsf{\Lambda}([e_2e_3,k_1k_1'],[e_1e_3,k_2k_2'],e_1e_2).$ 
Nous avons
sous les conditions
$k_ik_i'\mid e_i$ et $|\mu(e_1e_2)|=|\mu(e_3)|=1$, 
la relation
$\mathsf{\Lambda}({\ma{e},\ma{k} 
})=\mathsf{\Lambda}{([e_2e_3,k],[e_1e_3,k],e_1e_2) }$, 
avec $k=k_1k_1'k_2k_2'$. Ainsi la somme $T_{\ma{e},\ma{k} }(X)$ ne 
d\'epend que de $k$ et, \`a $k\mid
e_1e_2$ fix\'e, nous avons
$$\sum_{\colt{\ma{k}=(k_1,k_2,k_1',k_2')\in\N^4}{k_ik_i'=
(k,e_i)}}
\frac{\mu(k_1')\mu(k_2')}{2^{\omega(k_1)+\omega(k_2)}}
=\frac{\mu((k,e_1)) 
\mu((k,e_2))}{2^{\omega((k,e_1))+\omega((k,e_2))}}=\frac{\mu(k)}{2^{\omega(k)}}, 
$$
puisque le produit $e_1e_2$ est sans facteur carr\'e.

Nous avons donc
\begin{equation}\label{T(X)=sume}
T(X)=\sum_{\ma{e}\in\N^3} \mu(e_1e_2)\mu(e_3)\sum_{ k\mid e_1e_2}
\frac{\mu(k)}{2^{\omega(k)}} T_{\ma{e},k }(X),
\end{equation}
avec
$T_{\ma{e},k }(X)=S(X,\ma{d},\ma{D})$ 
et 
\begin{equation}\label{defdDparek}
\ma{d}=(e_2e_3,e_1e_3,e_1e_2),\quad
\ma{D}=([e_2e_3,k],[e_1e_3,k],e_1e_2).
\end{equation} 
Afin d'appliquer 
le Th\'eor\`eme \ref{maind}, nous
\'etudions chacune des quantit\'es qui interviennent dans l'\'enonc\'e.
Puisque $$e_2\mid (D_1,D_3),\qquad e_1\mid (D_2,D_3),\qquad 
[e_3,k]\mid (D_1,D_2) ,$$ le Lemme
\ref{lemmindelta} montre que
$$  \Big[\frac{e_1 }{(e_1 ,\Delta_{23} ) },\frac{ e_2}{ (e_2,
\Delta_{13})},\frac{[e_3,k ]}{([e_3,k ],\Delta_{12}) } \Big] \mid 
\delta(\ma{D}). 
$$   
Puisque  $e_1$ et $e_2$ sont premiers entre eux, nous obtenons
$$
  \frac{[ e_1e_2 , e_3,k ] }{([ e_1e_2 , e_3,k ],\Delta_{13} 
\Delta_{23}\Delta_{12}) }
\mid \Big[\frac{e_1e_2}{(e_1 ,\Delta_{23} )(e_2, 
\Delta_{13})},\frac{[e_3,k ]}{([e_3,k
],\Delta_{12}) } \Big]
\mid \delta(\ma{D}).$$
Enfin  la relation
$k\mid e_1e_2$
fournit
\begin{equation}\label{divdelta}
\frac{[ e_1e_2 , e_3  ] }{([ e_1e_2 , e_3  ],\Delta_{13} 
\Delta_{23}\Delta_{12}) }
\mid \delta(\ma{D}).
\end{equation}
Dor\'enavant, les formes $L_1$, $L_2$ et $Q$ sont fix\'ees et les 
constantes implicites dans les $O$
d\'ependent des coefficients de ces formes
et de supremum de la norme des \'el\'ements de $\mcal{R}$.  Nous pouvons donc  
supposer que $r'\ll r_\infty\ll 1$
d'apr\`es \eqref{prelim} et $a'(\ma{D}, \mDelta)\ll 1$. 

Pour montrer une majoration de $T_{\ma{e},k}(X)$, on utilise la 
valeur de $\delta(\ma{D})=\delta$
et on somme les $\x$ en fonction de la valeur de $(x_1,x_2)=\delta f$ 
avec $f\geq 1$. Il vient
\begin{align*}T_{\ma{e},k }(X)&\leq
\sum_{f\in \N}\sum_{\colt{\y\in \Z^2\cap (X/\delta f)\mcal{R}}{(y_1,y_2)=1}}
\tau(\delta fL_1(\y))\tau(\delta fL_2(\y))\tau(\delta ^2f^2Q(\y))
  \\&\leq\sum_{f\in \N}\tau(\delta f)^4
\sum_{\colt{\y\in \Z^2\cap (X/\delta 
f)\mcal{R}}{(y_1,y_2)=1}}\tau(L_1(\y))\tau( L_2(\y))\tau( Q(\y))
\\&\leq\sum_{f\in \N}\tau(\delta  )^4\tau(  f)^4
\sum_{\colt{\y\in \Z^2\cap (X/\delta 
f)\mcal{R}}{(y_1,y_2)=1}}\tau''(L_1(\y) L_2(\y) Q(\y))
,
\end{align*}
o\`u la fonction $\tau''$ a \'et\'e introduite en \eqref{deftau''}.
D'apr\`es le Corollaire 1 de \cite{nair}, la somme int\'erieure en $\y$ est
$\ll ( X/\delta f)^2(\log X)^3 $
ce qui fournit  d'apr\`es \eqref{divdelta}
\begin{equation*} 
T_{\ma{e},k }(X)
\ll  X^2(\log X)^3\frac{\tau(\delta  )^4}{\delta^2}
\ll  X^2(\log X)^3\frac{\tau([
e_1e_2 , e_3  ]  )^4}{[ e_1e_2 , e_3  ]^2}.
\end{equation*}
 La contribution des $\ma{e}$ dans la somme \eqref{T(X)=sume} 
d\'efinissant $T(X)$ tels que
$e=e_1e_2e_3\geq  (\log X)^2$ est donc major\'e pour tout 
$\varepsilon>0$ fix\'e par
\begin{align*}&\ll  X^2(\log X)^3\sum_{e\geq ( \log
X )^2}\frac{\tau(e)^{8}}{e^2}\sum_{m^2\mid e} m^2\cr&\ll
X^2(\log X)^3\sum_{m\in\N}\frac{\tau(m^2)^{8}}{m^2}\sum_{f\geq \max\{ 1,
(\log x)^2/m^2\}}
\frac{\tau(f)^{8}}{f^2}\cr&\ll
X^2(\log X)^{3+\varepsilon/2}\sum_{m\in\N}\frac{\tau(m^2)^{8}}{m^2+(\log
x)^2} 
\ll  X^2(\log
X)^{2+\varepsilon},
\end{align*}
o\`u $m=e/[e_1e_2,e_3]=(e_1e_2,e_3)$.
Lorsque $e<(\log X)^2$, le  Th\'eor\`eme
\ref{maind} fournit
\begin{align*} 
T_{\ma{e},k }(X) =& 2 \vol(\mcal{R})
  X^2(\log r'X)^3 \prod_{p}\sigma_p(\ma{d},\ma{D})
   +
O\Big(   \frac{e^{\varepsilon} X^2 (\log X)^{2}
\log\log X}{[ e_1e_2 , e_3  ]} \Big),
\end{align*}
o\`u, dans la somme $\sigma_p(\ma{d},\ma{D})$, le couple 
$(\ma{d},\ma{D}) $ est d\'efini par
\eqref{defdDparek}.
Comme $D_3'$ est sans facteur carr\'e, nous pouvons utiliser 
\eqref{majW} coupl\'e avec
\eqref{psi0sscarre}. Nous avons
\begin{align*} 
\prod_{p}\sigma_p(\ma{d},\ma{D})\ll D^{ \varepsilon}
\frac{\psi_0(D_1',D_2',D_3')}{D'}
& \ll \frac{e^{\varepsilon}(e_1e_2,e_3)k }{
[e_2e_3,k] [e_1e_3,k] e_1e_2}
\ll \frac{e^{\varepsilon}(e_1e_2,e_3)(e_3,k) }{
   (e_1e_2e_3)^2} .\end{align*}
Nous en d\'eduisons l'estimation
$$T(X)=2 \vol(\mcal{R})C
  X^2(\log X)^3+ O\big(         X^2 (\log X)^{2+\varepsilon}\big),$$
avec
$$C =
\sum_{\ma{e}\in\N^3} \mu(e_1e_2)\mu(e_3)\sum_{k\mid
e_1e_2}\frac{\mu(k)}{2^{\omega(k)}}\prod_{p}\sigma_p(\ma{d},\ma{D}). $$

\smallskip

Nous nous consacrons maintenant au calcul de $C$.
Pour des raisons de multiplicativit\'e, nous avons
$C=\prod_p\Big(1-\frac{1}{p}\Big)^3 C_p$ 
avec
\begin{equation*} 
C_p:=\!\!\!\!\!\!\sum_{\colt{\mve\in \{0,1\}^3}{\min
\{\varepsilon_1,\varepsilon_2\}=0}}\!\!\!\!\!\!(-1)^{\varepsilon_1+\varepsilon_2+\varepsilon_3}
\sum_{\colt{\mka}{0\leq\kappa_i\leq \varepsilon_i}}(-1/2)^{\kappa_1+\kappa_2}
\sum_{\mnu\in\Z_{\geq 0}^3} \overline\rho
( {N_1}, {N_2}, {N_3}),
\end{equation*}
avec 
\begin{align*} 
N_1 :=\max\{ \kappa_1,\varepsilon_2+\varepsilon_3+\nu_1  \},\quad
N_2 :=\max\{ \kappa_2,\varepsilon_1+\varepsilon_3+\nu_2  \} ,\quad
N_3 := \varepsilon_1+\varepsilon_2+\nu_3  
\end{align*}
et
$$ \overline\rho
( {n_1}, {n_2}, {n_3})=\overline\rho_p
( {n_1}, {n_2}, {n_3}):=\frac{ \rho
(p^{n_1},p^{n_2},p^{n_3})
}{p^{2n_1+2n_2+2n_3}}.$$
Pour simplifier les calculs nous fixons $p$ et n'indiquons plus la 
d\'ependance en $p$ de
$\overline\rho_p$. Nous intervertissons les sommations. Il vient
$C_p=\sum_{\mnu\in\Z_{\geq 0}^3}c_p(\mnu)$,  
avec
\begin{align*}
c_p(\mnu)
=&\overline\rho(\nu_1, \nu_2, \nu_3)-\overline\rho(\nu_1, \nu_2+1, \nu_3+1)
+\dm\overline\rho(\max\{\nu_1,1\}, \nu_2+1, \nu_3+1)\\
&-\overline\rho(\nu_1+1, \nu_2 , \nu_3+1)
+\dm\overline\rho(  \nu_1+1,\max\{\nu_2,1\}, \nu_3+1)
\\&-\overline\rho(\nu_1+1, \nu_2+1, \nu_3 )
+\dm\overline\rho(\nu_1+1, \nu_2+2, 
\nu_3+1)\\&+\dm\overline\rho(\nu_1+2, \nu_2+1, \nu_3+1).
\end{align*}
Soit $\delta_0$ la fonction caract\'eristique sur $\Z_{\geq 0}$ du point $0$.
Nous avons  
\begin{align*}\overline\rho(\max\{\nu_1,1\}, \nu_2+1, \nu_3+1) =&
\overline\rho(\nu_1, \nu_2+1, \nu_3+1)\\&+\delta_{0}(\nu_1)
\big(\rho(1, \nu_2+1, \nu_3+1)-\overline\rho(0, \nu_2+1, 
\nu_3+1)\big)\end{align*}
et
\begin{align*}\overline\rho(\nu_1+1, \max\{\nu_2,1\}, \nu_3+1) =&
\overline\rho(\nu_1+1, \nu_2, \nu_3+1)\\&+\delta_{0}(\nu_2)
\big(\rho(\nu_1+1,  1, \nu_3+1)-\overline\rho(\nu_1+1, 0, 
\nu_3+1)\big).\end{align*}
Nous avons
\begin{align*}
c_p(\mnu)
=&\overline\rho(\nu_1, \nu_2, \nu_3)-\dm\overline\rho(\nu_1, \nu_2+1, \nu_3+1)
\\&+\dm\delta_{0}(\nu_1)
\big(\rho(1, \nu_2+1, \nu_3+1)-\overline\rho(0, \nu_2+1, \nu_3+1)\big)
  \\
&-\dm\overline\rho(\nu_1+1, \nu_2 , 
\nu_3+1)
  -\overline\rho(\nu_1+1, \nu_2+1, \nu_3 )\\&
  +\dm\delta_{0}(\nu_2)
\big(\rho(\nu_1+1,  1, \nu_3+1)-\overline\rho(\nu_1+1, 0, 
\nu_3+1)\big)\\& +\dm\overline\rho(\nu_1+1, \nu_2+2,
\nu_3+1) +\dm\overline\rho(\nu_1+2, \nu_2+1, \nu_3+1).
\end{align*}
Nous observons des simplifications entre le premier et le cinqui\`eme 
terme  avec une contribution
\'egale \`a
\begin{align*}\sum_{\mnu\in\Z_{\geq 0}^3}\overline\rho(\nu_1, \nu_2, \nu_3)
& -\overline\rho(\nu_1+1, \nu_2+1, \nu_3 )
\\&=\sum_{(\nu ,\nu_3)\in\Z_{\geq 0}^2} \big(\overline\rho(0, \nu, \nu_3
) +\overline\rho(\nu  , 0,
\nu_3  )-\overline\rho(0  , 0,
\nu_3  )\big).\end{align*}
Nous observons des simplifications entre le  
deuxi\`eme et le
huiti\`eme terme avec une contribution
\'egale \`a
\begin{align*}
\sum_{\mnu\in\Z_{\geq 0}^3}&-\dm\overline\rho(\nu_1, \nu_2+1, \nu_3+1)
  + \dm\overline\rho(\nu_1+2, \nu_2+1, \nu_3+1)
\\&=-\dm\sum_{(\nu_2 ,\nu_3)\in\Z_{\geq 0}^2} \big(\overline\rho(0,
  \nu_2+1, \nu_3+1) +\overline\rho(1, \nu_2+1, \nu_3+1) \big)\end{align*}
de m\^eme entre le quatri\`eme et le septi\`eme.
Nous obtenons alors
\begin{align*}C_p =& \sum_{(\nu ,\nu_3)\in\Z_{\geq 0}^2} 
\big(\overline\rho(0, \nu, \nu_3
)-\overline\rho(0, \nu+1,
\nu_3+1 )\\&\qquad\qquad\qquad+\overline\rho(\nu  , 0,
\nu_3  )-\overline\rho(\nu+1,0,\nu_3+1)-\overline\rho(0, 0,\nu_3)\big)
\\=&\overline\rho(0,0, 0) +\sum_{ \nu  \in\N } \big(\overline\rho(0, \nu, 0)
+\overline\rho(\nu,0,0)+\overline\rho(0,0,\nu) \big).
\end{align*}
Cela cl\^ot la d\'emonstration.

\section{D\'emonstration   des corollaires}\lab{sectionCor}

\subsection{D\'emonstration du Corollaire \ref{corS*}.}

   Une inversion de Mob\"\i us fournit
\begin{align*}
   S^*(X,\ma{d},\ma{D};V) =\sum_{k\in\N}\mu(k)
S(X/k,\ma{d},\ma{D};kL_1,kL_2,k^2Q,\mcal{R};V ).
\end{align*}
Nous appliquons le Th\'eor\`eme \ref{maind}. Nous commen\c cons par 
estimer la contribution du terme
d'erreur.
La valeur $r'$ associ\'ee  satisfait
\begin{equation}\label{valrk}
r'(kL_1,kL_2,k^2
Q,k^2\mcal{R})=kr'(L_1,L_2,Q,\mcal{R}) \end{equation}
et donc l'ensemble $V$ est le m\^eme pour tous les $k$. 
Pour chaque valeur de $k$, la valeur de $a'$ associ\'ee est $\leq
a'(\ma{D},\mDelta)$. Le terme d'erreur est donc
$$\ll
(DL_\infty)^{\varepsilon} a'(\ma{D},\mDelta) (kr_\infty r'  +r_\infty^2)  \frac{ 
X^2  }{k^{2-\varepsilon}}(\log
X)^{2}
\log\log X.$$
La somme  de ce terme pour $k\leq K:= \log X $ est englob\'ee dans le 
terme d'erreur.
D'apr\`es le Lemme \ref{majSXdD}, la contribution des 
$k\in\,(K,r_\infty X]$ est major\'ee par
\begin{align*}&\ll (DL_\infty)^{\varepsilon}
a'(\ma{D},\mDelta)\Big(
        (r_\infty X)^2 (\log
X)^{3}\sum_{k>K}k^{-2+\varepsilon}+  (r_\infty 
X)^{1+\varepsilon}(\log X) \Big)\\&
\ll
(DL_\infty)^{\varepsilon}
a'(\ma{D},\mDelta)
        (r_\infty X)^2 (\log
X)^{ 2+\varepsilon}
\end{align*}
o\`u nous avons utilis\'e $r_\infty X\geq 1$.

Il reste \`a \'evaluer la contribution du terme principal qui, 
compte-tenu de~\eqref{valrk}, vaut
$$ 2 \vol(\mcal{R}) \vol(V)
  X^2(\log r'X)^3\sum_{k\leq K}\mu(k)W_k$$
avec
$$ 
W_k:=\frac{1}{k^2}\prod_{p}\Big(1-\frac{1}{p}\Big)^3\sum_{\mnu\in\Z_{\geq
0}^3}  \overline{\rho}_\kappa (p^{N_1},p^{N_2},p^{N_3}) .$$
Ici, nous avons $\kappa=v_p(k)$ et
$$
\overline{\rho}_\kappa(p^{N_1},p^{N_2},p^{N_3})=\frac{\#\{ \x\in 
[0,p^{N_1+N_2+N_3})\,:\, p^{N_i}\mid
p^\kappa L_i,
\,p^{N_3}\mid p^{2\kappa}Q  \}}{p^{2N_1+2N_2+2N_3}},
$$
avec les notations \eqref{defNi}.
Nous pouvons consid\'erer la somme enti\`ere (ie  
sans la 
contrainte $k\leq K$) puisque,
d'apr\`es  \eqref{majW}, 
$W_k\ll (DL_\infty)^\varepsilon  k^{-2+\varepsilon}$.  
Nous faisons les m\^emes manipulations qu'\`a la fin de la 
d\'emonstration du  Th\'eor\`eme~\ref{maind}
sans indiquer les d\'etails. Nous obtenons
$$\sum_{k\in\N}\mu(k)W_k=\prod_pw^*_p$$ avec
\begin{align*}w^*_p&:=\Big(1-\frac{1}{p}\Big)^3\sum_{\mnu\in\Z_{\geq 
0}^3}\sum_{\kappa\in\{0,1\}}
  (-1)^\kappa\frac{\overline{\rho}_\kappa(p^{N_1},p^{N_2},p^{N_3})}{p^{2\kappa}}
=\sigma^*_p(\ma{d},\ma{D})
\end{align*}
avec $\sigma^*_p(\ma{d},\ma{D})$ d\'efini en \eqref{defsigma*d} et 
\eqref{defsigma*d0}.
Lorsque $v_p(D)=0$, nous avons utilis\'e le fait que
$$\sum_{\kappa\in\{0,1\}}
  (-1)^\kappa\frac{\overline{\rho}_\kappa(1,1,1)}{p^{2\kappa}}
=1-\frac{1}{p^{2}} .$$ Cela ach\`eve la preuve du Corollaire \ref{corS*}.

\subsection{D\'emonstration du Corollaire \ref{corT*}. }
Pour un entier   $n\in\N$, nous d\'esi\-gnons par $\mcal{P}(n)$ l'ensemble des
facteurs premiers de $n$. Nous utilisons la
notation $\mcal{S}=\mcal{P}(\Delta_{12}\Delta_{13}\Delta_{23})$.
Nous aurons besoin du lemme suivant.

\begin{lem}\lab{lemg} Soit $g$ une fonction multiplicative
satisfaisant \eqref{conditiong} pour un $\eta_0$ fix\'e.
Il existe $\eta_1\in(0,\eta_0)$ tel que  
$$ 
\sum_{\colt{\ma{m}, \ma{s}\in\N^3}{\mcal{P}(m_i s_i)\subseteq \mcal{S}}}
  \frac{|\mu(s_1)\mu(s_2)\mu(s_3)||g(m_1m_2m_3)|}{(
    m_1m_2m_3s_1s_2s_3)^{1/2-\eta_1 }}\ll 1.
$$  soit born\'ee.
\end{lem}

Nous n'indiquons pas   la preuve qui est imm\'ediate. Notons
seulement que comme $\mcal{S}$ est fix\'e, les $s_i$ qui sont sans 
facteur carr\'e sont born\'es.

\smallskip

Nous sommes maintenant en mesure de d\'emontrer le Corollaire 
\ref{corT*}. Introduisons
  les entiers $m_i$
enti\`erement d\'efinis par
\begin{equation}\label{conditionS}\mcal{P}(L_i(\x)/m_i)\cap\mcal{S}=\mcal{P}(Q(\x)/m_3)\cap\mcal{S}=\emptyset,
\qquad \mcal{P}(m_i)\subseteq
\mcal{S} .
\end{equation}
Puisque $(x_1,x_2)=1$, les  $L_i(\x)/m_i$ et $Q(\x)/m_3$ sont 
premiers deux \`a deux. Nous \'ecrivons
\begin{align*} 
&
g(L_1(\x), L_2(\x), Q(\x);V) \\
&\quad=\sum_{\colt{k\mid
m_1m_2m_3}{k_i=(k,m_i)}}(1*h)(k)
\sum_{\colt{q_i\mid L_i(\x)/m_i,\quad q_3\mid Q(\x)/m_3}{
\big(\frac{\log k_1q_1}{\log r'X},\frac{\log k_2q_2}{\log 
r'X},\frac{\log k_3q_3}{2\log r'X}\big)\in
V}}\prod_{j=1}^3(1*h)(q_j)\\
\\&\quad=\sum_{\colt{k\mid
m_1m_2m_3}{k_i=(k,m_i)}}(1*h)(k)
\sum_{\tolt{\ma{d}\in\N^3}{(d_1,d_2,d_3)=1}{\mcal{P}(d_i)\cap 
\mcal{S}=\emptyset}}
h(d_1 d_2 d_3)\!\!\!\!\!\!\!\!
\sum_{\colt{\ell_id_i\mid L_i(\x)/m_i,\quad \ell_3d_3\mid Q(\x)/m_3}{
\big(\frac{\log k_1\ell_1d_i}{\log r'X},\frac{\log k_2\ell_2d_2}{\log 
r'X},\frac{\log k_3\ell_3d_3}{2\log
r'X}\big)\in V}} 1,
\end{align*}
o\`u dans cette formule les $m_j$   satisfont les conditions 
\eqref{conditionS}. 
Nous pouvons alors \'ecrire  
$$ 
\Big(\frac{\log k_1\ell_1d_1}{\log r'X},\frac{\log 
k_2\ell_2d_2}{\log r'X},\frac{\log
k_3\ell_3d_3}{2\log r'X}\Big)\in V
\Leftrightarrow 
\Big(\frac{\log  \ell_1 }{\log r'X},\frac{\log \ell_2 }{\log 
r'X},\frac{\log  \ell_3 }{2\log
r'X}\Big)\in V(\ma{k},\ma{d}),
$$ 
avec
$\vol(V(\ma{k},\ma{d}))=\vol(V)+O(\log (md)/\log X).$
Nous avons donc
\begin{align*} 
T_g^*(X;V)=&\sum_{\colt{\ma{m} \in\N^3}{ \mcal{P}(m_j)\subseteq
\mcal{S}}}\sum_{\colt{k\mid
m_1m_2m_3}{k_j=(k,m_j)}}(1*h)(k)\sum_{\tolt{\ma{d}\in\N^3}{(d_1,d_2,d_3)=1}{\mcal{P}(d_j)\cap
\mcal{S}=\emptyset}} h(d_1 d_2 d_3)
\\&\times
\sum_{\tolt{\x\in \mathsf{\Lambda}({\ma{d}})\cap X\mcal{R}}{(x_1,x_2)=1
}{\mcal{P}((L_1L_2Q)(\x)/m_1m_2m_3
)\cap\mcal{S}= \emptyset  }}
\!\!\!\!\!\!\!\!\!\!\!\!\!\!\!\!
\tau\Big(\frac{L_1(\x)}{m_1d_1},\frac{L_2(\x)}{m_2d_2},\frac{Q(\x) 
}{m_3d_3},V(\ma{k},\ma{d}) \Big)
\\
=&\sum_{\colt{\ma{m} \in\N^3}{ \mcal{P}(m_j)\subseteq
\mcal{S}}}\sum_{\colt{k\mid
m_1m_2m_3}{k_j=(k,m_j)}}(1*h)(k)\sum_{\tolt{\ma{d},\ma{s}\in\N^3}{(d_1,d_2,d_3)=1}{\mcal{P}(d_j)\cap
\mcal{S}=\emptyset,\mcal{P}( s_j)\subseteq \mcal{S}}} h(d_1 d_2 
d_3)\\
&\times
\mu(s_1)\mu(s_2)\mu(s_3)
S^*(X,\ma{d}',\ma{D}'; V(\ma{k},\ma{d})),
\end{align*}
avec
$$\ma{d}'=(m_1d_1,m_2d_2, m_3d_3),\qquad 
\ma{D}'=(s_1m_1d_1,s_2m_2d_2,s_3m_3d_3).$$

Nous choisisssons $M=(\log X)^{2/\eta_1}$ avec $\eta_1$ issu du Lemme 
\ref{lemg}.
Nous  appliquons le Corollaire \ref{corS*} lorsque $d_i,s_i,m_i\leq 
M$ alors que nous faisons appel au
Lemme \ref{majSXdD} sinon.
Rappelons que les $s_i$   sont born\'es donc ils peuvent \^etre pris $\leq M$.
La contribution du premier terme du majorant donn\'e par le Lemme 
\ref{majSXdD}  avec le choix
$\varepsilon<\eta_1/3$ des
$\ma{m},\ma{s},\ma{d}$ tels que
$\max_i\{d_i, m_i\}> M$ est donc
\begin{align*}
&\leq
\sum_{\tolt{\ma{m},\ma{s},\ma{d}\in\N^3}{(d_1,d_2,d_3)=1=|\mu(s_i)|}{\mcal{P}(d_i)\cap
\mcal{S}=\emptyset,\mcal{P}(m_i s_i)\subseteq \mcal{S}}}
\frac{(mds)^{\eta_1/2}}{\log X}|g(m) h(d)
S^*(X,\ma{d}',\ma{D}')|
\\&\ll X^2(\log X)^{2+\varepsilon}\sum_{d\in\N}\frac{|h(d)|}{d^{1-\eta_1}}
\!\!\!\!
\sum_{\tolt{\ma{m}, \ma{s}\in\N^3}{ |\mu(s_i)|=1}{ \mcal{P}(m_i s_i)\subseteq
\mcal{S}}}
\frac{ |g(m)|}{(ms)^{1/2-\eta_1 }}
  \ll X^2(\log X)^{2+\varepsilon}
\end{align*}
avec $m=m_1m_2m_3$, $s=s_1s_2s_3$ et $d=d_1 d_2 d_3$. Ici, nous avons
utilis\'e la majoration $\tau(d)^2\ll d^{\eta_1/6}.$

Nous utilisons la relation $md\ll X$ sous la forme $1\ll 
(X/md)^{1-3\varepsilon}$. La contribution du
second terme du majorant donn\'e par le Lemme
\ref{majSXdD}  des
$\ma{m},\ma{s},\ma{d}$ est
\begin{align*} 
\ll X^{1+\varepsilon} 
\hspace{-0.2cm}
\sum_{\tolt{\ma{m}, 
\ma{d}\in\N^3}{(d_1,d_2,d_3)=1 }{\mcal{P}(d_i)\cap
\mcal{S}=\emptyset,\mcal{P}(m_i  )\subseteq \mcal{S}}}
|g(m) h(d) | (md)^{\varepsilon}
&\ll X^{2-\varepsilon} \sum_{\colt{ m, d\in\N^3}{\mcal{P}(m  )\subseteq 
\mcal{S}}}
\frac{|g(m) h(d) |}{ (md)^{1 -4\varepsilon}}
\ll X^{2-\varepsilon},
\end{align*}
pourvu que $1-4\varepsilon\geq\dm-\eta_0$.

Puisque 
$\vol(V(\ma{k},\ma{d}))=\vol(V)+O(\log\log X/\log X),$
la contribution du terme principal du  Corollaire \ref{corS*} est
$$ 2 \vol(\mcal{R})\vol(V)
  X^2(\log r'X)^3C^*$$
avec
$$C^*=\!\!\!\!
\sum_{\tolt{\ma{m},\ma{s},\ma{d}\in\N^3}{(d_1,d_2,d_3)=1 }{\mcal{P}(d_i)\cap
\mcal{S}=\emptyset,\mcal{P}(m_i s_i)\subseteq 
\mcal{S}}}\!\!\!\!\!\!\!\!\!\!\!\!\!\!\!
g(m_1m_2m_3)  h(d_1 d_2 d_3)\mu(s_1)\mu(s_2)\mu(s_3)
W(\ma{s},\ma{m},\ma{d}).$$
Ici, nous avons \'etendu la somme au $\max_i\{d_i,s_i,m_i\}> M$ en 
majorant $W(\ma{s},\ma{m},\ma{d})$
gr\^ace \`a  
  \eqref{majW}. La m\^eme m\'ethode
que ci-dessus fournit un terme d'erreur acceptable.

Il reste \`a v\'erifier la valeur de $C^*$.
Posons
$$\overline{\rho}^*(p^{\nu_1},p^{\nu_2},p^{\nu_3})=
\frac{\rho^*(p^{\nu_1},p^{\nu_2},p^{\nu_3})}{p^{2\nu_1+2\nu_2+2\nu_3}}
\qquad (\nu_1+\nu_2+\nu_3\geq 1)$$
et $\overline{\rho}^*(1,1,1)=1-1/p^2$.
En \'ecrivant la somme d\'efinissant $C^*$
sous forme de produit eul\'erien, nous obtenons
$$C^*=\prod_p\Big(1-\frac{1}{p}\Big)^3C^*_p$$
avec $C^*_p $ ayant deux expressions suivant que $p\in \mcal{S}$.
Lorsque $p\in \mcal{S}$, nous avons
$$C^*_p=\sum_{\mmu\in\Z_{\geq 0}^3}g(p^{\mu_1+\mu_2+\mu_3})
\sum_{\msigma\in\{0,1\}^3}(-1)^{\sigma_1+\sigma_2+\sigma_3} \sum_{\mnu\in
\Z_{\geq 0}^3}\overline{\rho}^*(p^{N_1},p^{N_2},p^{N_3})$$
avec la relation
$N_i=\max\{ \sigma_i+\mu_i ,\mu_i+ \nu_i\} $ alors que lorsque 
$p\notin \mcal{S}$ nous avons
$$C^*_p=\sum_{\colt{\mdelta\in\Z_{\geq 0}^3}{\#\{i\,:\, \delta_i\geq 
1\}\leq 1}}h(p^{\delta_1
+\delta_2+\delta_3})  \sum_{\mnu\in
\Z_{\geq 
0}^3}\overline{\rho}^*(p^{\delta_1+\nu_1},p^{\delta_2+\nu_2},p^{\delta_3+\nu_3}).$$
Lorsque $p\notin \mcal{S}$, nous avons  
$\overline{\rho}^*(p^{\nu_1},p^{\nu_2},p^{\nu_3})=0$  
quand $\#\{i\,:\, \nu_i\geq 1\}\geq 2$.  
Nous obtenons apr\`es changement d'indices
\begin{align*}
C^*_p&= \sum_{\mnu'\in
\Z_{\geq 0}^3}\overline{\rho}^*(p^{ \nu_1'},p^{\nu_2'},p^{ \nu_3'})
  \sum_{\tolt{\mdelta\in\Z_{\geq 0}^3}{\#\{i\,:\, \delta_i\geq 1\}\leq 
1}{0\leq \delta_i\leq \nu_i'}}
h(p^{\delta_1
+\delta_2+\delta_3})
\\&=\sum_{\mnu'\in
\Z_{\geq 0}^3}\overline{\rho}^*(p^{ \nu_1'},p^{\nu_2'},p^{ \nu_3'})
  \sum_{\colt{\mdelta\in\Z_{\geq 0}^3}{0\leq \delta_i\leq \nu_i'}}
h(p^{\delta_1
+\delta_2+\delta_3})
\\&=\sum_{\mnu'\in
\Z_{\geq 0}^3}(g*\mu)(p^{ \nu_1'+\nu_2'+ \nu_3'})\overline{\rho}^*(p^{ 
\nu_1'},p^{\nu_2'},p^{
\nu_3'}),
\end{align*}
o\`u nous avons utilis\'e la relation $h*1=g*\mu$.
Gr\^ace \`a \eqref{ovrhoP}, nous obtenons
\begin{align*}
C^*_p& =
\sum_{\colt{\mnu \in
\Z_{\geq 0}^3 }{\nu_1 +\nu_2 + \nu_3\geq 1}}
\!\!\!\!\!\!\!\!
g (p^{ \nu_1 +\nu_2 + \nu_3 })\sum_{ \msigma\in\{0,
1\}^3}(-1)^{\sigma_1+\sigma_2+\sigma_3}\overline{\rho}^*(p^{ 
\nu_1+\sigma_1},p^{\nu_2+\sigma_2},p^{
\nu_3+\sigma_3})
\\&= \sum_{\mnu \in
\Z_{\geq 0}^3 }
g (p^{ \nu_1 +\nu_2 + \nu_3 })
\overline{\rho}^\dagger_p(\nu_1,\nu_2,\nu_3).
\end{align*}

Lorsque $p\in \mcal{S}$, nous obtenons apr\`es changement d'indices
$$C^*_p=\sum_{\ma{N}\in\Z_{\geq 0}^3}\overline{\rho}^*(p^{N_1},p^{N_2},p^{N_3})
\sum_{\colt{\mmu\in\Z_{\geq
0}^3}{\mu_i\leq N_i}}g(p^{\mu_1+\mu_2+\mu_3})
\!\!\!\!\!\!\!\!\!\!\sum_{\colt{\mnu\in\Z_{\geq
0}^3,\msigma\in\{0,1\}^3}{\max\{\nu_i,\sigma_i\}=N_i-\mu_i}}\!\!\!\!\!\!(-1)^{\sigma_1+\sigma_2+\sigma_3}.
  $$
ce qui fournit encore gr\^ace \`a \eqref{ovrhoP} et une interversion 
des sommations
\begin{align*}
C^*_p&= \sum_{\mnu \in
\Z_{\geq 0}^3 }
g (p^{ \nu_1 +\nu_2 + \nu_3 })
\overline{\rho}^\dagger_p(\nu_1,\nu_2,\nu_3).
\end{align*}

Notons
$\mu_{\mnu}(n)$ la fonction multiplicative en $n$ d\'efinie, lorsque 
$r=\#\{ i\,:\, \nu_i\geq 1\}$, par
\begin{equation*} 
\mu_{\mnu} ( p^n )=\left\{\begin{array}{ll}
\mu ( p^n ), &\mbox{si  $n=0$ ou $r\leq 1$},\\
-r  ,  &\mbox{si $n=1$, $r\geq 2$},
\\
1  ,  &\mbox{si $n=2$, $r=2$},\\
3  ,  &\mbox{si $n=2$, $r=3$},\\
0  ,  &\mbox{si $n\geq 3$, $r= 2$},\\
-1  ,  &\mbox{si $n=3$, $r=3$},
\\
0 ,  &\mbox{si $n\geq 4$}.
\end{array}
\right.
\end{equation*}
  Lorsque $p\nmid \Delta_{12}\Delta_{13}\Delta_{23}$, la condition 
$\rho^*(p^{\nu_1},p^{\nu_2},p^{
\nu_3})\neq 0$ implique $\mu_{\mnu}(p^n)=\mu(p^n)$ et ainsi nous avons
$$\qquad (g*\mu_{\mnu})(p^{\nu_1+\nu_2+\nu_3}) 
=(1*h)(p^{\nu_1+\nu_2+\nu_3})   .$$

Or
$$\sum_{\colt{\mnu,\msigma\in\Z_{\geq
0}^3}{\max\{\nu_i,\sigma_i\}=N_i-\mu_i}}(-1)^{\sigma_1+\sigma_2+\sigma_3}=
\mu(p^{N_1-\mu_1})\mu(p^{N_2-\mu_2})\mu(p^{N_3-\mu_3}),$$
et lorsque $\nu_j\leq N_j$ nous avons
$$\mu_{\ma{N}} ( p^{\nu_1+\nu_2+\nu_3}
)=\sum_{\mu_1+\mu_2+\mu_3=\nu_1+\nu_2+\nu_3}\mu(p^{N_1-\mu_1})\mu(p^{N_2-\mu_2})\mu(p^{N_3-\mu_3}).$$
Il est facile alors d'en d\'eduire la formule
\begin{equation*} 
C_p^*  =  \sum_{ \mnu\in\Z_{\geq 0}^3} 
(g*\mu_{\mnu})(p^{\nu_1+\nu_2+\nu_3})
  \frac{\rho^*(p^{\nu_1},p^{\nu_2},p^{ 
\nu_3})}{p^{2\nu_1+2\nu_2+2\nu_3}} .
\end{equation*}

\subsection{D\'emonstration du Corollaire \ref{corT'}. }

Nous d\'eduisons le Corollaire \ref{corT'} du Corollaire \ref{corT*} 
via une int\'egration par parties.
Sa forme n'\'etant pas usuelle, nous fournissons quelques d\'etails.
La formule
$$\frac{1}{\max\{|x_1|,|x_2|\}^2}=2\int_{\max\{|x_1|,|x_2|\}}^{X}\frac{\d 
t}{t^3}+
\frac{1}{X^2}$$
implique apr\`es une interversion de sommation
l'estimation
   $$ T_g'(X;V') =2\int_{1}^{X}\sum_{\colt{\x\in \Z^2\cap
t\mcal{R}}{(x_1,x_2)=1}} g'(L_1(\x), L_2(\x), Q(\x);V')
\frac{\d t}{t^3}+ O\big(   (\log X)^{3}
\big) .$$
Puisque la contribution des ${\x\in \Z^2\cap
t/(\log t)\mcal{R}}$ est major\'ee par $O((\log X)^{3})$, on peut 
remplacer la condition
$$\Big(\frac{\log d_1}{\log  Y},\frac{\log d_2}{\log  Y},\frac{\log 
d_3}{2\log  Y},\frac{\log
\max\{|x_1|,|x_2|\}}{ \log  Y}\Big)\in V'$$
par
$$\Big(\frac{\log d_1}{\log  r't},\frac{\log d_2}{\log 
r't},\frac{\log d_3}{2\log
r't} \Big)\in V_{t,Y} $$
avec
$$V_{t,Y}:=\Big\{\ma{t}\in [0,1]^3\,:\quad 
\Big(t_1,t_2,t_3,1+O\Big(\frac{\log\log (3t) }{ \log
(3t)}\Big)\Big)\in
\frac{
\log  Y}{\log  r't }(V'\cap V_0')\Big\}.$$ Ici, nous avons utilis\'e 
la relation $\log r'\ll 1$.
Puisque
$$\vol(V_{t,Y})=\Big(\frac{
\log  Y}{\log  r't }\Big)^3\int_{(t_1,t_2,t_3,\log t/\log Y)\in 
V'\cap V_0'} \d t_1\d t_2\d t_3,$$
le Corollaire \ref{corT*} fournit alors
\begin{align*} T_g'(X;V') &=2\int_{1}^{X}\sum_{\colt{\x\in \Z^2\cap
t\mcal{R}}{(x_1,x_2)=1}} g (L_1(\x), L_2(\x), Q(\x);V_{t,Y})
\frac{\d t}{t^3}+ O\big(   (\log X)^{3}
\big) \\&=
4 C^* \vol(\mcal{R})   (\log    Y)^3\int_{1}^{X} 
\int_{(t_1,t_2,t_3,\log t/\log Y)\in V'\cap V_0'} \d
t_1\d t_2\d t_3\frac{\d t}{t }\\&\quad+ O\big(   (\log X)^{3+\varepsilon}
\big)\\
&=
4 C^* \vol(\mcal{R})\vol(V'\cap V_0'(\log X/\log Y))   (\log    Y)^4+ 
O\big(   (\log X)^{3+\varepsilon}
\big)
.\end{align*}

\end{document}